\documentclass[10.5pt,a4paper]{article}

\usepackage{graphicx,latexsym,euscript,makeidx,color,bm}
\usepackage{amsmath,amsfonts,amssymb,amsthm,thmtools,mathtools,mathrsfs,enumerate}
\usepackage[colorlinks,linkcolor=blue,citecolor=red]{hyperref}
\usepackage{geometry}
\allowdisplaybreaks[4]
   \def\cA{{\cal A}}  \def\hu{\hat{u}}
 \def\sB{\mathscr{B}}    
   \def\cC{{\cal C}}  
     
\def\dbE{\mathbb{E}}     
\def\dbF{\mathbb{F}} \def\sF{\mathscr{F}}  \def\cF{{\cal F}}  
     
\def\dbH{\mathbb{H}}     
   \def\cI{{\cal I}}  
     
   \def\cK{{\cal K}}  \def\bX{\bar X_{\scT}}
   \def\cL{{\cal L}}  \def\bY{\bar Y_{\scT}}
   \def\cM{{\cal M}}  \def\bZ{\bar Z_{\scT}}
   \def\cN{{\cal N}}  \def\bu{\bar u_{1,\scT}}
   \def\cO{{\cal O}}  \def\bv{\bar u_{2,\scT}}
\def\dbP{\mathbb{P}}   \def\cP{{\cal P}}
   
\def\dbR{\mathbb{R}}   
\def\dbS{\mathbb{S}}   
   
 \def\sU{\mathscr{U}}

   \def\cX{{\cal X}}
   \def\cY{{\cal Y}}
   
\def\ss{\smallskip}   \def\lt{\left}          \def\hb{\hbox}
\def\ms{\medskip}     \def\rt{\right}         \def\ae{\text{a.e.}}
\def\h{\hat}          \def\lan{\langle}       \def\as{\text{a.s.}}
\def\q{\quad}         \def\ran{\rangle}       
\def\qq{\qquad}             \def\1n{\negthinspace }
\def\no{\noindent}          \def\2n{\1n\1n}
\def\hp{\hphantom}         \def\scp{\scriptscriptstyle}
\def\nn{\nonumber}         \def\scT{\scp T}
\def\rf{\eqref}       \def\Blan{\Big\lan}     
\def\cd{\cdot}        \def\Bran{\Big\ran}     
\def\deq{\triangleq}  \def\({\Big(}           
       \def\){\Big)}           \def\les{\leqslant}
   \def\[{\Big[}           \def\ges{\geqslant}
  \def\]{\Big]}           %\def\tc{\textcolor{red}}
\def\wh{\widehat}                             %\def\tb{\textcolor{blue}}

\def\a{\alpha}       \def\l{\lambda}    
\def\b{\beta}               \def\F{\varPhi}
\def\d{\delta}           \def\G{\varGamma}
\def\e{\varepsilon}      \def\L{\varLambda}
\def\f{\varphi}             \def\Om{\varOmega}
\def\g{\gamma}       \def\si{\sigma}    \def\Si{\varSigma}
\def\i{\infty}             \def\Th{\varTheta}
\def\k{\kappa}                          

\def\bde{\begin{definition}\label}
\def\ede{\end{definition}}
\def\bel{\begin{equation}\label}
\def\ee{\end{equation}}
\def\bt{\begin{theorem}\label}
\def\et{\end{theorem}}
\def\bc{\begin{corollary}\label}
\def\ec{\end{corollary}}
\def\bl{\begin{lemma}\label}
\def\el{\end{lemma}}
\def\bp{\begin{proposition}\label}
\def\ep{\end{proposition}}
\def\ba{\begin{array}}
\def\ea{\end{array}}

\newtheoremstyle{thry}% name
{}      % Space above
{}      % Space below
{\sl}   % Body font
{}      % Indent amount
{\bf}   % Theorem head font
{.}     % Punctuation after theorem head
{.5em}  % Space after theorem head
{}      % Theorem head spec (can be left empty, meaning 'normal')
\theoremstyle{thry}

\newtheorem{theorem}{Theorem}[section]
\newtheorem{proposition}[theorem]{Proposition}
\newtheorem{corollary}[theorem]{Corollary}
\newtheorem{lemma}[theorem]{Lemma}

\theoremstyle{definition}
\newtheorem{definition}[theorem]{Definition}

\theoremstyle{remark}
\newtheorem{remark}[theorem]{Remark}

\def\punct{}
\newtheoremstyle{dotless}{}{}{\rm}{}{\bf}{\punct}{.5em}{}
\theoremstyle{dotless}

\makeatletter
   
   \@addtoreset{equation}{section}
   \newcommand{\setword}[2]{%
   \phantomsection
   #1\def\@currentlabel{\unexpanded{#1}}\label{#2}%
   }
\makeatother

\begin{document}
\title{\bf Long-Time Behavior of Zero-Sum Linear-Quadratic Stochastic Differential Games}
\author{Jingrui Sun\thanks{Department of Mathematics and SUSTech International Center for Mathematics,
                           Southern University of Science and Technology, Shenzhen, Guangdong,
                           518055, China (Email: sunjr@sustech.edu.cn).
                           This author is supported in part by NSFC grants 12322118 and 12271242, and
                           Shenzhen Fundamental Research General Program JCYJ20220530112814032. }
~~~and~~
Jiongmin Yong\thanks{Department of Mathematics, University of Central Florida, Orlando, FL 32816, USA
                    (Email: jiongmin.yong@ucf.edu).
                    This author is supported in part by NSF grant DMS-2305475. }
}

\maketitle

\no{\bf Abstract.}
The paper investigates the long-time behavior of zero-sum linear-quadratic stochastic
differential games, aiming to demonstrate that, under appropriate conditions, both the
saddle strategy and the optimal state process exhibit the exponential turnpike property.
Namely, for the majority of the time horizon, the distributions of the saddle strategy
and the optimal state process closely stay near certain (time-invariant) distributions
$\nu_1^*$, $\nu_2^*$ and $\mu^*$, respectively.
Additionally, as a byproduct, we solve the infinite horizon version of the differential
game and derive closed-loop representations for its open-loop saddle strategy, which has
not been proved in the literature.

\ms
\no{\bf Key words.}
Stochastic differential game, zero-sum, linear-quadratic, saddle strategy,
turnpike property, Riccati equation.

\ms
\no{\bf AMS 2020 Mathematics Subject Classification.}
91A05, 91A15, 49N10, 49N70.

\section{Introduction}\label{Sec:Intro}

Let $(\Om,\sF,\dbF,\dbP)$ be a complete filtered probability space satisfying the usual conditions. Suppose that a standard one-dimensional Brownian motion
$W=\{W(t);\,t\ges 0\}$ is defined on this space with $\dbF\equiv\{\sF_t\}_{t\ges0}$ being its natural filtration augmented by all the $\dbP$-null sets in $\sF$. Consider the controlled linear stochastic
differential equation (SDE, for short)
\begin{equation}\label{state}\left\{\begin{aligned}
dX(t) &= [AX(t)+B_1u_1(t)+B_2u_2(t)+b]dt \\
&\hp{=\ } +[CX(t)+D_1u_1(t)+D_2u_2(t)+\si]dW(t),\q 0\les t\les T,\\
X(0) &=x
\end{aligned}\right.\end{equation}
and the quadratic performance functional
\begin{align}\label{J_T}
J_{\scT}(x;u_1(\cd),u_2(\cd)) \deq \dbE\!\int_0^T\!\bigg[\Blan\!\!
   \begin{pmatrix}Q   & \!\!S_1^\top & \!\!S_2^\top \\
                  S_1 & \!\!R_{11}   & \!\!R_{12}   \\
                  S_2 & \!\!R_{21}   & \!\!R_{22}   \end{pmatrix}\!\!
   \begin{pmatrix}X(t) \\ u_1(t) \\ u_2(t)\end{pmatrix}\!,\!
   \begin{pmatrix}X(t) \\ u_1(t) \\ u_2(t)\end{pmatrix}\!\!\Bran \!+\! 2\Blan\!\!
   \begin{pmatrix}  q  \\ r_1    \\ r_2   \end{pmatrix}\!,\!
   \begin{pmatrix}X(t) \\ u_1(t) \\ u_2(t)\end{pmatrix}\!\!\Bran\bigg]dt,
\end{align}
where in state equation \rf{state}, the coefficients
$$ A,C,\in\dbR^{n\times n}, \q  B_i,D_i\in\dbR^{n\times m_i}, \q  b,\si\in\dbR^n,\q(i=1,2) $$
are constant matrices/vectors, and in cost functional \rf{J_T}, the weighting coefficients
$$ Q\in\dbR^{n\times n}, \q S_i\in\dbR^{m_i\times n}, \q R_{ij}\in\dbR^{m_i\times m_j},
\q q\in\dbR^n, \q r_i\in\dbR^{m_i}, \q (i,j=1,2) $$
are also constant matrices, with the square block matrix in \rf{J_T} being symmetric.
The superscript $\top$ in \rf{J_T} denotes the transpose of matrices,
and $\lan\cd\,,\cd\ran$ stands for the usual Euclidean inner product of vectors 
(whose induced norm is denoted by $|\cd|$). For simplicity of notation, we write $\f(\cd)\in\dbF$ if a process $\f(\cd)$ is
$\dbF$-progressively measurable. Let
\begin{align}\label{sU[0,T]}
\sU_i[0,T] \deq \Big\{\f:[0,T]\times\Om\to\dbR^{m_i}\bigm|\f(\cd)\in\dbF~\hb{and}
~\dbE\int_0^T|\f(t)|^2dt<\i\Big\}, \q i=1,2.
\end{align}
It is clear that for each initial state $x\in\dbR^n$ and control pair $(u_1(\cd),u_2(\cd))\in\sU_1[0,T]\times\sU_2[0,T]\equiv\sU[0,T]$, the
state equation \rf{state} admits a unique square-integrable solution
$X(\cd)\equiv X(\cd\,;x,u_1(\cd),u_2(\cd))$.
Thus, the corresponding cost functional \rf{J_T} is well-defined.
The {\it zero-sum linear-quadratic (LQ, for short) stochastic differential game}
over the finite time horizon $[0,T]$ can be stated as follows.

\ms

{\bf Problem (DG)$_{\scT}$.}
For any initial state $x\in\dbR^n$, find a pair
$(\bu(\cd),\bv(\cd))\in\sU_1[0,T]\times\sU_2[0,T]$ such that
\begin{align*}
J_{\scT}(x;\bu(\cd),u_2(\cd))
\les J_{\scT}(x;\bu(\cd),\bv(\cd)) \equiv V_{\scT}(x)
\les J_{\scT}(x;u_1(\cd),\bv(\cd)), \\
\forall(u_1(\cd),u_2(\cd))\in\sU_1[0,T]\times\sU_2[0,T].
\end{align*}

\ss

From the above, we see that Player 1 is the minimizer (by taking $u_1(\cd)\in\sU_1[0,T]$)
and Player 2 is the maximizer (by taking $u_2(\cd)\in\sU_2[0,T]$).
The pair $(\bu(\cd),\bv(\cd))$ (if it exists) is called an {\it open-loop saddle strategy}
of Problem (DG)$_{\scT}$ at the initial state $x$, $\bX(\cd)$ is called the
corresponding {\it open-loop optimal state process}, and $V_{\scT}(\cd)$ is called the
{\it value function} of the game.
We also refer to $(\bX(\cd),\bu(\cd),\bv(\cd))$ as an {\it open-loop optimal triple}
at $x\in\dbR^n$.
When such a triple exists for every initial state $x$, we say that Problem (DG)$_{\scT}$
is {\it open-loop solvable}. In the case that $b,\si,q,r_1,r_2$ all vanish, we denote the corresponding game by
Problem (DG)$_{\scT}^{\scp0}$ and call it a {\it homogenous differential game} on $[0,T]$.
The performance functional and value function of Problem (DG)$_{\scT}^{\scp0}$ are denoted by $J_{\scT}^{\scp0}(x;u_1(\cd),u_2(\cd))$ and $V_{\scT}^{\scp0}(\cd)$, respectively.

\ms

It is clear that in the above, the open-loop saddle strategy (if it exists) is seemingly anticipating.
Namely, in determining the value $(\bu(t),\bv(t))$ of $(\bu(\cd),\bv(\cd))$ at $t\in[0,T)$,
some future information on the triple $(\bX(s),\bu(s),\bv(s))$ for $s\in[t,T]$ will be used.
This can be seen, at least, from the optimality conditions.
Thus, from this point of view, the above open-loop saddle strategy is not practically realizable.
To get practical feasibility, one could introduce the set of state-feedback controls.
For any $(\Th_i(\cd),v_i(\cd))\in L^\i(0,T;\dbR^{m_i\times n})\times\sU_i[0,T]$, ($i=1,2$),
the closed-loop system reads:
\begin{equation*}\left\{\begin{aligned}
dX(t) &= \big\{[A+B_1\Th_1(t)+B_2\Th_2(t)]X(t) +[B_1v_1(t)+B_2v_2(t)+b]\big\}dt \\
&\hp{=\ }+\big\{[C+D_1\Th_1(t)+D_2\Th_2(t)]X(t) +[D_1v_1(t)+D_2v_2(t)+\si]\big\}dW(t), \q 0\les t\les T,\\
X(0) &=x.
\end{aligned}\right.\end{equation*}
Correspondingly, we denote
\begin{align*}
& J_{\scT}\big(x;(\Th_1(\cd),v_1(\cd)),(\Th_2(\cd),v_2(\cd))\big)
\deq J_{\scT}\big(x;\Th_1(\cd)X(\cd)+v_1(\cd),\Th_2(\cd)X(\cd)+v_2(\cd)\big)\\
&\q=\dbE\int_0^T\bigg[\Blan\!
    \begin{pmatrix}Q   & \!\!S_1^\top & \!\!S_2^\top \\
                   S_1 & \!\!R_{11}   & \!\!R_{12}   \\
                   S_2 & \!\!R_{21}   & \!\!R_{22}   \end{pmatrix}\!\!
    \begin{pmatrix}X(t) \\ \Th_1(t)X(t)+v_1(t) \\ \Th_2(t)X(t)+v_2(t)\end{pmatrix}\!,\!
    \begin{pmatrix}X(t) \\ \Th_1(t)X(t)+v_1(t) \\ \Th_2(t)X(t)+v_2(t)\end{pmatrix}\!\Bran\\
&\q\hp{=\ } +2\Blan\!\begin{pmatrix} q \\ r_1 \\ r_2 \end{pmatrix}\!,\!
            \begin{pmatrix}X(t)\\ \Th_1(t)X(t)+v_1(t)\\ \Th_2(t)X(t)+v_2(t)\end{pmatrix}\!\Bran\bigg]dt.
\end{align*}
The above amounts to saying that the controls take the following forms:
$$ u_i(t)=\Th_i(t)X(t)+v_i(t),\q t\in[0,T]; \q i=1,2. $$
Clearly, the above controls are anti-participating, meaning they do not rely on future information.
Consequently, they are practically feasible, in principle.
Now, a four-tuple $(\bar\Th_{1,\scT}(\cd),\bar v_{1,\scT}(\cd)$;
$\bar\Th_{2,\scT}(\cd),\bar v_{2,\scT}(\cd))$ is called a {\it closed-loop saddle strategy}
of Problem (DG)$_{\scT}$, if for any initial state $x\in\dbR^n$ and any
$(\Th_i(\cd),v_i(\cd))\in L^\i(0,T;\dbR^{m_i\times n})\times\sU_i[0,T]$; $i=1,2$, the following holds:
\begin{align*}
J_{\scT}\big(x;(\bar\Th_{1,\scT}(\cd),\bar v_{1,\scT}(\cd)),(\Th_{2,\scT}(\cd),v_{2,\scT}(\cd))\big)
&\les J_{\scT}\big(x;(\bar\Th_{1,\scT}(\cd),\bar v_{1,\scT}(\cd)),(\bar\Th_{2,\scT}(\cd),\bar v_{2,\scT}(\cd))\big)\\
&\les J_{\scT}\big(x;(\Th_1(\cd),v_1(\cd)),(\bar\Th_{2,\scT}(\cd),\bar v_{2,\scT}(\cd))\big).
\end{align*}

\ss

Recently, extensive research has been conducted on LQ stochastic differential games
and their various extensions in the literature.
In \cite{Mou-Yong2006}, Mou--Yong approached Problem (DG)$_{\scT}$ from an open-loop
perspective using the Hilbert space method.
Sun--Yong \cite{Sun-Yong2014} characterized open-loop and closed-loop saddle strategies
for Problem (DG)$_{\scT}$ and established their certain properties.
Subsequently, Sun \cite{Sun2021} furthered the study of Problem (DG)$_{\scT}$, revealing
fundamental properties of this class of games and illustrating differences between stochastic and deterministic cases. Bardi--Priuli \cite{Bardi-Priuli2014} delved into ergodic nonzero-sum LQ stochastic differential games with $N$ players, while Duncan \cite{Duncan2014} investigated a class
of zero-sum LQ stochastic differential games with the noise process being an arbitrary
square-integrable stochastic process with continuous sample paths.
Sun--Yong--Zhang \cite{Sun-Yong-Zhang2016} tackled the infinite horizon version of
Problem (DG)$_{\scT}$, followed by additional work of Li--Shi--Yong \cite{Li-Shi-Yong2021} incorporating
mean-field terms.
Moon \cite{Moon2021,Moon2023} explored LQ stochastic leader-follower Stackelberg differential
games for jump-diffusion and Markov jump-diffusion systems.
Yu--Zhang--Zhang \cite{Yu-Zhang-Zhang2022} extended the framework of Problem (DG)$_{\scT}$
to include Poisson jumps.
Additionally, numerous other works on LQ differential games exist,
including \cite{Hamadene1998,Hamadene1999} on nonzero-sum LQ games
and \cite{Bensoussan-Sung-Yam-Yung2016,Caines-Huang2021} on mean-field
LQ games for large-population systems.

\ms

In this paper, we will assume certain conditions that guarantee that for any given
time horizon $[0,T]$, $u_1(\cd)\mapsto J_{\scT}(x;u_1(\cd),u_2(\cd))$ is uniformly
convex and $u_2(\cd)\mapsto J_{\scT}(x;u_1(\cd),u_2(\cd))$ is uniformly concave.
In this case, it is known the following facts:
(i) For any given initial state, Problem (DG)$_{\scT}$ admits a unique open-loop
saddle strategy;
(ii) Problem (DG)$_{\scT}$ admits a unique closed-loop saddle strategy,
determined by the solution to the corresponding differential Riccati equation,
and a terminal value problem of ordinary differential equation (ODE, for short);
(iii) The outcome of the closed-loop saddle strategy corresponding to an initial
state is the open-loop saddle strategy (for that given initial state);
and (iv) For a given initial state, the open-loop saddle strategy admits a unique
closed-loop representation, and it coincides with the outcome of the closed-loop
saddle strategy corresponding to that initial state. See Sun-Yong \cite{Sun-Yong2020book-b} for details.

\ms

From the above-mentioned facts, we see that for the open-loop saddle strategy at any initial state, one has its closed-loop representation, or regarded it as the outcome of the closed-loop saddle
strategy (corresponding to the initial state).
Thus, the open-loop saddle strategy can always be thought of initial state dependent and time
anti-participating.

\ms

In this paper, we are going to investigate the asymptotic behavior of Problem (DG)$_{\scT}$ as $T\to\i$.
It turns out that under appropriate conditions, the optimal triple of Problem (DG)$_{\scT}$ exhibits the
so-called {\it exponential turnpike property}.
Specifically, let $(\bX(\cd),\bu(\cd),\bv(\cd))$ be the optimal triple of Problem (DG)$_{\scT}$
for the initial state $x$, with $\mu_{\scT}(t;x)$, $\nu_{\sc1,\scT}(t;x)$, and $\nu_{\sc2,\scT}(t;x)$
representing the corresponding distributions of $\bX(t)$, $\bu(t)$, and $\bv(t)$, respectively.
The exponential turnpike property asserts the existence of unique probability distributions
$\mu^*$, $\nu_1^*$, and $\nu_2^*$ (on $\dbR^m$, $\dbR^{m_1}$, and $\dbR^{m_2}$, respectively),
independent of $x$ and $T$, such that for some constants $K,\l>0$,
\begin{equation}\label{fenbu-TP}\begin{aligned}
& d(\mu_{\scT}(t;x),\mu^*) + d(\nu_{\sc1,\scT}(t;x),\nu_1^*) + d(\nu_{\sc2,\scT}(t;x),\nu_2^*) \\
&\q \les K(|x|^2+1)\[e^{-\l t}+e^{-\l(T-t)}\], \q\forall t\in[0,T],
\end{aligned}\end{equation}
where $d$ is the Wasserstein distance on the set of distributions (see later for a definition).
We have misused the notation a little here, since the random variables could be valued in the space
of different dimensions, which can be identified from the context.
In what follows, $K,\l>0$ are generic constants which could be different from line to line.

\ms

Let $\k\in(0,1)$ be an arbitrary number. Then inequality \rf{fenbu-TP} implies that
$$ d(\mu_{\scT}(t;x),\mu^*) + d(\nu_{1,\scT}(t;x),\nu_1^*) + d(\nu_{2,\scT}(t;x),\nu_2^*)
\les 2K(|x|^2+1)e^{-\l\k T}, \q\forall t\in[\k T,(1-\k)T]. $$
Since $K$ and $\l$ are independent of $T$, when the time horizon $[0,T]$ is very large,
the distributions of the optimal triple $(\bX(\cd),\bu(\cd),\bv(\cd))$ remain very close
to the time-invariant and $x$-independent distributions $\mu^*$, $\nu_1^*$, and $\nu_2^*$, respectively, for most of $[0,T]$. Consequently, we could use $\mu^*$ as the initial distribution of the state equation and $(\nu_1^*,\nu_2^*)$ as the law of the control processes of the players to approximately solve Problem (DG)$_{\scT}$.
Additionally, we will show that the invariant distribution $\mu^*$ is determined by a stable SDE, or it can be constructed by solving a stationary Fokker-Planck equation, according to the classical theory of SDEs, and $\nu_1^*$ and $\nu_2^*$ are determined by $\mu^*$.

\ms

The turnpike property, initially realized by Ramsey \cite{Ramsey1928} and
von Neumann \cite{Neumann1945} in the early part of the last century, originated from
the study of optimal solutions to dynamic optimization problems with an infinite time
horizon in the context of economic growth. The term ``turnpike" was coined by Dorfman,
Samuelson, and Solow \cite{Dorfman-Samuelson-Solow1958} in 1958, inspired by a similar
feature observed in toll highways in the United States.
In recent years, significant progress has been made in addressing deterministic optimal
control problems, as evidenced by studies such as
\cite{Porretta-Zuazua2013,Damm-Grune-Stieler-Worthmann2014,Trelat-Zuazua2015,Zaslavski2019,
Trelat-Zhang2018,Grune-Guglielmi2018,Lou-Wang2019,Breiten-Pfeiffer2020,Sakamoto-Zuazua2021}
and the references cited therein. However, in the realm of stochastic optimal control problems,
the investigation of corresponding turnpike properties is relatively nascent.
To the best of our knowledge, the work of \cite{Sun-Wang-Yong2022} marked the first
attempt to uncover such properties for LQ stochastic optimal control problems,
followed by a more comprehensive and generalized study presented in \cite{Sun-Yong2024}.

\ms

As previously mentioned, this paper aims to investigate the turnpike property for
zero-sum LQ stochastic differential games.
In this context, the primary contributions and challenges of our paper can be
summarized as follows:

\ms

(i) Establishing the turnpike property requires analyzing the behavior of the solution
$P_{\scT}(t)$ to the associated differential Riccati equation as the time horizon $T$
tends to infinity. In the optimal control case \cite{Sun-Wang-Yong2022,Sun-Yong2024},
under the stabilizability condition of the state equation over $[0,\i)$, positive
definiteness conditions are imposed on the weighting coefficients of the performance
functional, which implies the monotonicity of $T\mapsto P_{\scT}(t)$.
Now, such a monotonicity is lost due to the opposing roles played by the weighting
coefficients for the two players. This poses significant challenges, and we have to
seek a different approach.

\ms

(ii) The exponential turnpike property \rf{fenbu-TP} is established for the zero-sum
LQ stochastic differential game in the distributional sense.
Since in reality people primarily care about the distributions of stochastic processes,
this result suggests that we may use $(\nu_1^*,\nu_2^*)$ as an approximate solution
to Problem (DG)$_{\scT}$, which is more convenient since it is independent of time and
the initial state.
Thus, our result offers more potential applications in practical settings.

\ms

(iii) As a byproduct and an intermediate step in the analysis of the turnpike
property for Problem (DG)$_{\scT}$, we solve the infinite horizon version of
Problem (DG)$_{\scT}^{\scp0}$, denoted by Problem (DG)$_{\scp\i}^{\scp0}$,
under the stability condition. We prove that the associated algebraic Riccati equation is uniquely solvable, under our assumed conditions,
ensuring that Problem (DG)$_{\scp\i}^{\scp0}$ is both open-loop and closed-loop
solvable, and that the open-loop saddle strategy admits a closed-loop representation.
Recall that in \cite{Sun-Yong-Zhang2016} (see also \cite{Sun-Yong2020book-b}),
it was only proved the equivalence of closed-loop solvability and the algebraic
Riccati equation's solvability. No sufficient conditions were given for the
closed-loop solvability of Problem (DG)$_{\scp\i}^{\scp0}$.

\ms

The rest of the paper is organized as follows. In \autoref{Sec:Pre}, we introduce some frequently used notation and present some preliminary results. Problem (DG)$_{\scp\i}^{\scp0}$ is introduced in \autoref{Sec:Ihorizon-game}, together with some fundamental analysis on the problem. In \autoref{Sec:Ric}, we investigate the asymptotic behavior of the solution $P_{\scT}(\cd)$ to the differential Riccati equation associated with Problem (DG)$_{\scT}$. Finally, we establish the exponential turnpike property for Problem (DG)$_{\scT}$ in \autoref{Sec:TP}.

\section{Preliminaries}\label{Sec:Pre}

We start by introducing some frequently used notation.
Let $\dbR^{n\times m}$ be the space of $n\times m$ real matrices equipped with
the Frobenius inner product. Denote by $\dbS^n$ the subspace of $\dbR^{n\times n}$
consisting of symmetric matrices and by $\dbS^n_+$ the subset of $\dbS^n$ consisting
of positive definite matrices.
For $\dbS^n$-valued functions $M(\cd)$ and $N(\cd)$, we write $M(\cd)\ges N(\cd)$
(respectively, $M(\cd)>N(\cd)$) if $M(\cd)-N(\cd)$ is positive semidefinite
(respectively, positive definite) almost everywhere with respect to the Lebesgue measure.
The identity matrix of size $n$ is denoted by $I_n$ (or simply $I$  when no confusion arises),
and a vector is always considered as a column vector unless otherwise specified.
For a Euclidian space $\dbH$ (which could be $\dbR^n$, $\dbR^{n\times m}$, etc.),
we define (recalling that $\f(\cd)\in\dbF$ means that the process $\f(\cd)$ is progressively
measurable with respect to $\dbF$)
\begin{align*}
& C([0,T];\dbH)\deq\Big\{\f:[0,T]\to\dbH\bigm|\f~\hb{is continuous}\Big\},\\
& L^\i(0,T;\dbH)\deq\Big\{\f:[0,T]\to\dbH \bigm|\f~\hb{is Lebesgue measurable
and essentially bounded}\Big\}, \\
& L_{\cF_{\scT}}^2(\Om;\dbH)\deq\Big\{\xi:\Om\to\dbH\bigm|\xi~\hb{is $\cF_{\scT}$-measurable and}
~\dbE|\xi|^2<\i\Big\},\\
& L_{\dbF}^2(0,T;\dbH)\deq\Big\{\f:[0,T]\times\Om\to\dbH\bigm|\f(\cd)\in\dbF~\hb{and}~
\dbE\int_0^T|\f(t)|^2dt<\i\Big\},\\
& L_{\dbF}^2(0,\i;\dbH)\deq\Big\{\f:[0,\i)\times\Om\to\dbH\bigm|\f\in\dbF ~\hb{and}~
\dbE\int_0^\i|\f(t)|^2dt<\i\Big\}.
\end{align*}
In what follows, we will denote $m=m_1+m_2$, and
\begin{align}\label{Notation:BDSR}
B=(B_1,B_2), \q D=(D_1,D_2), \q S=\begin{pmatrix}S_1 \\ S_2\end{pmatrix}, \q
R=\begin{pmatrix}R_{11} & R_{12} \\ R_{21}& R_{22}\end{pmatrix},\q
r=\begin{pmatrix}r_1 \\ r_2\end{pmatrix}.
\end{align}
Also, we let $\lan\,\cd\,,\cd\,\ran$ be the inner product in various spaces which can be
identified from the context. Next, we introduce the following hypothesis.

\ms

{\bf(A1)} There exists a constant $\d>0$ such that for every $T>0$,
\begin{equation}\label{convex-concave}\left\{\begin{aligned}
& J_{\scT}^{\scp0}(0;u_1(\cd),0)\ges\d\,\dbE\int_0^T|u_1(t)|^2dt,  ~&\forall u_1(\cd)\in\sU_1[0,T], \\
& J_{\scT}^{\scp0}(0;0,u_2(\cd))\les-\d\,\dbE\int_0^T|u_2(t)|^2dt, ~&\forall u_2(\cd)\in\sU_2[0,T].
\end{aligned}\right.\end{equation}

\ss

For convenience, we refer to \rf{convex-concave} as the {\it uniform convexity/concavity condition}.
The following result essentially is taken from \cite{Sun2021} and \cite{Sun-Yong2020book-b}.

\begin{theorem}\label{thm:uT-rep}
Let {\rm(A1)} hold. Then the following hold:

\ms

{\rm(i)} For any initial state $x\in\dbR^n$, Problem (DG)$_{\scT}$ has a unique
open-loop saddle strategy.

\ms

{\rm(ii)} Let $\bar u_{\scT}(\cd)\equiv\begin{pmatrix}\bu(\cd) \\ \bv(\cd)\end{pmatrix}
\in\sU_1[0,T]\times\sU_2[0,T]$. Let $\bX(\cd)$ be the corresponding state process with
initial state $x$ and $(\bY(\cd),\bZ(\cd))$ the adapted solution to the following backward
stochastic differential equation (BSDE, for short):
\begin{equation}\label{BSDE}\left\{\begin{aligned}
d\bY(t) &= -[A^\top\bY(t)+C^\top\bZ(t)+Q\bX(t)+S^\top\bar u_{\scT}(t)+q]dt + \bZ(t)dW(t), \\
 \bY(T) &= 0.
\end{aligned}\right.\end{equation}
Then $(\bu(\cd),\bv(\cd))$ is the unique open-loop saddle strategy of Problem (DG)$_{\scT}$
for $x$ if and only if the following stationarity condition holds:
\begin{equation}\label{cdn:pingwen}
B^\top\bY(t)+D^\top\bZ(t)+S\bX(t)+R\bar u_{\scT}(t)+r=0, \q\ae~t\in[0,T],~\as
\end{equation}

\ss

{\rm(iii)} Problem (DG)$_{\scT}$ has a unique closed-loop saddle strategy
$(\bar\Th_{\scT}(\cd),\bar v_{\scT}(\cd))$.

\ms

{\rm(iv)} The following differential Riccati equation
\begin{equation}\label{Ric:game}\left\{\begin{aligned}
& \dot P_{\scT} + P_{\scT}A + A^\top P_{\scT} + C^\top P_{\scT}C + Q \\
&\hp{P_{\scT}} -(P_{\scT}B + C^\top P_{\scT}D + S^\top)(R + D^\top P_{\scT}D)^{-1}
                (B^\top P_{\scT} + D^\top P_{\scT}C + S)=0, \\
& P_{\scT}(T)=0
\end{aligned}\right.\end{equation}
admits a unique {\it strongly regular solution} $P_{\scT}(\cd)\in C([0,T];\dbS^n)$,
which means that for some constant $\a>0$,
\begin{align}\label{R+DPD}
(-1)^{i+1}[R_{ii}+D_i^{\top}P_{\scT}(t)D_i] \ges \a I, \q\forall t\in[0,T];\q i=1,2,
\end{align}
(implicitly implying the invertibility of the matrix $R+D^\top P_{\scT}(t)D$) and
that \rf{Ric:game} holds.

\ms

{\rm(v)} Let $P_{\scT}(\cd)\in C([0,T];\dbS^n)$ be the strongly regular solution of
\rf{Ric:game} and set
\begin{align}\label{def:barTh}
\bar\Th_{\scT}(t) \deq -[R+D^\top P_{\scT}(t)D]^{-1}[B^\top P_{\scT}(t)+D^\top P_{\scT}(t)C+S],
\q t\in[0,T].
\end{align}
The following terminal value problem of an ODE admits a unique solution $\f_{\scT}(\cd)$:
\begin{equation}\label{ODE:phi}\left\{\begin{aligned}
& \dot\f_{\scT}(t) + [A\!+\!B\bar\Th_{\scT}(t)]^\top\f_{\scT}(t)
  + [C\!+\!D\bar\Th_{\scT}(t)]^\top P_{\scT}(t)\si + \bar\Th_{\scT}(t)^\top r + P_{\scT}(t)b + q =0, \\
&\f_{\scT}(T)=0.
\end{aligned}\right.\end{equation}
If we set
\begin{align}\label{def:vT}
\bar v_{\scT}(t)\deq-[R+D^\top P_{\scT}(t)D]^{-1}[B^\top\f_{\scT}(t)+D^\top P_{\scT}(t)\si+r], \q t\in[0,T],
\end{align}
then the unique closed-loop saddle strategy of Problem (DG)$_{\scT}$ is given by
$(\bar\Th_{\scT}(\cd),\bar v_{\scT}(\cd))$.

\ms

{\rm(vi)} For any initial state $x\in\dbR^n$, let $\bX(\cd)$ be the solution
of the state equation \rf{state} under the state feedback
$$ \bar u_{\scT}(t) \equiv \begin{pmatrix}\bar u_{1,\scT}(t) \\ \bar u_{2,\scT}(t)\end{pmatrix}
\deq \bar\Th_{\scT}(t)\bX(t)+\bar v_{\scT}(t), \q t\in[0,T], $$
which is called the outcome of $(\bar\Th_{\scT}(\cd),\bar v_{\scT}(\cd))$ corresponding to $x$.
Then $\bar u_{\scT}(\cd)$ is the open-loop saddle strategy Problem (DG)$_{\scT}$ for $x$.
\end{theorem}

\section{The Infinite Horizon Problem}\label{Sec:Ihorizon-game}

In establishing the turnpike property for Problem (DG)$_{\scT}$, a crucial step
is to demonstrate the exponential convergence of the solution $P_{\scT}(\cd)$
to the differential Riccati equation \rf{Ric:game} as $T\to\i$.
Note that in the current case, as mentioned in the introduction, we do no have
the monotonicity of $T\mapsto P_{\scT}(\cd)$.
Thus, the approach used for optimal control problems like in \cite{Sun-Wang-Yong2022}
and \cite{Sun-Yong2024} do not apply.
To tackle this challenge, we need to carefully investigate Problem (DG)$_{\scp\i}^{\scp0}$,
and conduct some fundamental analysis, making full use of the structure of the problem. We now carry out this.

\ms

Consider the following state equation, i.e., \rf{state} with $b=\si=0$, and $T=\i$:
\begin{equation}\label{state:IH}\left\{\begin{aligned}
dX(t)&= [AX(t)\1n+\1n B_1u_1(t)\1n+\1n B_2u_2(t)]dt
        +[CX(t)\1n+\1n D_1u_1(t)\1n+\1n D_2u_2(t)]dW(t), \q t\ges0, \\
 X(0)&= x
\end{aligned}\right.\end{equation}
and the following quadratic performance functional over $[0,\i)$, i.e., in \rf{J_T}, $r=0$ and $q=0$:
\begin{align}\label{cost:IH}
J_{\scp\i}^{\scp0}(x;u_1(\cd),u_2(\cd)) \deq \dbE\int_0^\i \Blan\!\!
   \begin{pmatrix}Q   & \!\!S_1^\top & \!\!S_2^\top \\
                  S_1 & \!\!R_{11}   & \!\!R_{12}   \\
                  S_2 & \!\!R_{21}   & \!\!R_{22}   \end{pmatrix}\!\!
   \begin{pmatrix}X(t) \\ u_1(t) \\ u_2(t)\end{pmatrix}\!,\!
   \begin{pmatrix}X(t) \\ u_1(t) \\ u_2(t)\end{pmatrix}\!\!\Bran dt.
\end{align}
Similar to \rf{sU[0,T]}, we can define $\sU_i[0,\i)$, $i=1,2$,
and $\sU[0,\i)\deq\sU_1[0,\i)\times\sU_2[0,\i)$. We call a control pair
$$(u_1(\cd),u_2(\cd))\in\sU_1[0,\i)\times\sU_2[0,\i),$$
{\it admissible} for the initial state $x$ if the corresponding state process
$X(\cd\,;x,u_1(\cd),u_2(\cd))$ is square-integrable over $[0,\i)$, i.e.,
$$ \dbE\int_0^\i|X(t;x,u_1(\cd),u_2(\cd))|^2dt<\i. $$
Clearly, the performance functional \rf{cost:IH} is well-defined for such control pairs.
We denote the set of admissible control pairs for the initial state $x$ by $\sU_{ad}(x)$.

\ms

In general, $\sU_{ad}(x)$ depends on the initial state $x$ and is only a subset of
$\sU_1[0,\i)\times\sU_2[0,\i)$. However, it can be shown that
\begin{align}\label{U=U}
\sU_{ad}(x)=\sU_1[0,\i)\times\sU_2[0,\i), \q\forall x\in\dbR^n,
\end{align}
under the {\it $L^2$-stability} condition, which we now recall.

\begin{definition}
The system
$$dX(t)=AX(t)dt+CX(t)dW(t),$$
denoted by $[A,C]$, is called {\it $L^2$-stable} if its solution $X(\cd\,;x)$
with initial state $x$ satisfies
\begin{align}\label{L2-stable}
\dbE\int_0^\i|X(t;x)|^2dt<\i, \q\forall x\in\dbR^n.
\end{align}
\end{definition}

We now present the following result, from which it is not difficult to see
that \rf{U=U} holds when the system $[A,C]$ is $L^2$-stable.

\begin{proposition}\label{prop:SDE-bound}
Suppose that $[A,C]$ is $L^2$-stable. Then there exist constants $K,\l>0$ such that
for any $b(\cd),\si(\cd)\in L^2_\dbF(0,\i;\dbR^n)$ and any $x\in\dbR^n$, the solution
$X(\cd)\equiv X(\cd\,;x)$ to the SDE
$$\left\{\begin{aligned}
dX(t) &= [AX(t)+b(t)]dt+[CX(t)+\si(t)]dW(t), \q t\ges0,\\
 X(0) &= x
\end{aligned}\right.$$
satisfies the following estimates:
\begin{align}
\dbE|X(t)|^2 &\les K\lt[e^{-\l t}|x|^2+\dbE\int_0^t\(|b(s)|^2+|\si(s)|^2\)ds\rt],
\q\forall t\ges0, \label{Es:1} \\
\dbE\int_0^t|X(s)|^2ds &\les K\lt[|x|^2+\dbE\int_0^t\(|b(s)|^2+|\si(s)|^2\)ds\rt],
\q\forall t\ges0. \label{Es:2}
\end{align}
\end{proposition}

\begin{proof}
Since $[A,C]$ is $L^2$-stable, by \cite[Theorem 3.2.3]{Sun-Yong2020book-a},
there exists a positive definite matrix $P\in\dbS^n_+$ such that
$$ PA + A^\top P + C^\top PC + 2I_n = 0. $$
Applying It\^o's rule to $t\mapsto\lan PX(t),X(t)\ran$, we obtain
\begin{align*}
{d\over dt}\dbE\lan PX(t),X(t)\ran
&=\dbE\[\lan(PA+A^\top P+C^\top PC)X(t),X(t)\ran\\
&\hp{=\ } +2\lan Pb(t)+C^\top P\si(t),X(t)\ran + \lan P\si(t),\si(t)\ran\] \\
&=\dbE\[-2|X(t)|^2+2\lan Pb(t)+C^\top P\si(t),X(t)\ran+\lan P\si(t),\si(t)\ran\].
\end{align*}
Let $\g>0$ be the largest eigenvalue of $P$ and set
$$ \l\deq\g^{-1}, \q \a(t) \deq Pb(t)+C^\top P\si(t), \q \b(t) \deq \lan P\si(t),\si(t)\ran; \q t\ges 0. $$
Then we have
\begin{align*}
& -|X(t)|^2\les-\l\lan PX(t),X(t)\ran,\\
& 2\lan Pb(t)+C^\top P\si(t),X(t)\ran\les |X(t)|^2+|\a(t)|^2,
\end{align*}
and hence
\begin{align*}
{d\over dt}\dbE\lan PX(t),X(t)\ran
&\les\dbE\[-|X(t)|^2+|\a(t)|^2+\b(t)\] \\
&\les-\l\dbE\lan PX(t),X(t)\ran+\dbE\big[|\a(t)|^2+\b(t)\big].
\end{align*}
By Gronwall's inequality in differential form,
\begin{align*}
\dbE\lan PX(t),X(t)\ran
&\les \lan Px,x\ran e^{-\l t} +\int_0^te^{-\l(t-s)}\dbE\[|\a(s)|^2+\b(s)\]ds\\
&\les |P|e^{-\l t}|x|^2 + K\dbE\int_0^t\[|b(s)|^2+|\si(s)|^2\]ds, \q\forall t\ges0,
\end{align*}
and
\begin{align*}
\dbE\int_0^t\lan PX(s),X(s)\ran ds
&\les|P||x|^2\int_0^te^{-\l s}ds + K\int_0^t\int_0^se^{-\l(s-r)}\dbE\[|b(r)|^2+|\si(r)|^2\]drds \\
&\les K\(|x|^2+\dbE\int_0^t\[|b(s)|^2+|\si(s)|^2\]ds\),\q\forall t\ges0.
\end{align*}
Since $P>0$, the desired estimates can be readily obtained.
\end{proof}

As a consequence of \autoref{prop:SDE-bound}, the following result provides
an equivalent statement for the $L^2$-stability, commonly referred to as the
{\it mean-square exponential stability} of $[A,C]$.

\begin{corollary}\label{crllry:stability}
System $[A,C]$ is $L^2$-stable if and only if there exist constants $K,\l>0$ such that
\begin{align}\label{e-stable}
\dbE|X(t;x)|^2 \les K e^{-\l t}|x|^2, \q\forall t\ges 0,~\forall x\in\dbR^n.
\end{align}
\end{corollary}

\begin{proof}
Clearly, \rf{e-stable} implies \rf{L2-stable}.
On the other hand, if $[A,C]$ is $L^2$-stable, applying \rf{Es:1} to the case
$b(\cd),\si(\cd)=0$ yields the mean-square exponential stability of $[A,C]$.
\end{proof}

In our game setting over the interval $[0,\i)$, the performance functional needs to
be well-defined, which is essential for ensuring the well-formulation of
Problem (DG)$_{\scp\i}^{\scp0}$, and exhibit uniform convexity in $u_1(\cd)$ and uniform concavity in $u_2(\cd)$, which, corresponds to (A1), guarantees the existence of a unique open-loop saddle strategy. In order to satisfy the former condition, any admissible control pair must consist of two parts: The first part is to cooperatively stabilize the state equation (assuming the homogeneous state equation is stabilizable). The second part involves choosing suitable controls from $\sU_i[0,\i)$ to achieve the goal of minimization/maximization. To avoid additional technicalities, we directly assume that the state system is stable to ensure that the performance functional is well-defined over $\sU_1[0,\i)\times\sU_2[0,\i)$, rather than requiring stabilizability as assumed in the optimal control case. Therefore, we introduce the following hypothesis.

\ms

{\bf(A2)} System $[A,C]$ is $L^2$-stable.

\ms

Similar to the finite horizon case, we introduce the notion of open-loop saddle strategy
for Problem (DG)$_{\scp\i}^{\scp0}$.

\begin{definition}
A control pair $(\bar u_1(\cd),\bar u_2(\cd))\in\sU_1[0,\i)\times\sU_2[0,\i)$ is called
an {\it open-loop saddle strategy} of Problem (DG)$_{\scp\i}^{\scp0}$ for the initial
state $x$ if
\begin{align*}
J(x;\bar u_1(\cd),u_2(\cd)) \les J(x;\bar u_1(\cd),\bar u_2(\cd))\les J(x;u_1(\cd),\bar u_2(\cd)),\\
\forall (u_1(\cd),u_2(\cd))\in\sU_1[0,\i)\times\sU_2[0,\i).
\end{align*}
\end{definition}

To establish the uniqueness and existence of an open-loop saddle strategy for Problem (DG)$_{\scp\i}^{\scp0}$,
let us consider, for $i=1,2$, the following SDE:
\begin{equation}\label{SDE:Xi}\left\{\begin{aligned}
dX_i(t) &=[AX_i(t) + B_iu_i(t)]dt + [CX_i(t) + D_iu_i(t)]dW(t), \q t\in[0,T],  \\
 X_i(0) &= 0.
\end{aligned}\right.\end{equation}
Since $[A,C]$ is $L^2$-stable, by \autoref{prop:SDE-bound}, there exists a constant $K>0$,
independent of $T$, such that the solution $X_i(\cd)$ of \rf{SDE:Xi} satisfies
\begin{equation}\label{Bound:Xi}
\dbE\int_0^T|X_i(t)|^2dt \les K \dbE\int_0^T|u_i(t)|^2dt, \q\forall u_i(\cd)\in\sU_i[0,T],\q\forall T>0.
\end{equation}
By the linearity of \rf{SDE:Xi}, the linear operators
$$ \cL_{i,\scT}:\sU_i[0,T]\to L_{\dbF}^2(0,T;\dbR^n),  \q
\wh\cL_{i,\scT}:\sU_i[0,T]\to L_{\cF_T}^2(\Om;\dbR^n); \q i=1,2, $$
defined by
\begin{equation}\label{def:cL}
[\cL_{\sc i,\scT}u_i](\cd)\deq X_i(\cd), \q\wh\cL_{\sc i,\scT} u_i \deq X_i(T); \q i=1,2,
\end{equation}
where $X_i(\cd)$ is the solution of \rf{SDE:Xi} corresponding to $u_i(\cd)$, are bounded
uniformly in $T$. Similarly, the linear operators
$$\cN_{\scT}:\dbR^n\to L_\dbF^2(0,T;\dbR^n),\q\wh\cN_{\scT}:\dbR^n\to L_{\cF_T}^2(\Om;\dbR^n)$$
defined by
$$[\cN_{\scT}x](\cd)\deq X_0(\cd), \q\wh\cN_{\scT}x\deq X_0(T),$$
where $X_0(\cd)$ is the solution of
$$\left\{\begin{aligned}
dX_0(t) &= AX_0(t)dt + CX_0(t)dW(t), \q t\in[0,T],\\
 X_0(0) &= x,
\end{aligned}\right.$$
are also bounded, uniformly in $T$. Further, it is easily seen that the solution of
$$\left\{\begin{aligned}
dX(t) &= [AX(t)+B_1u_1(t)+B_2u_2(t)]dt + [CX(t)+D_1u_1(t)+D_2u_2(t)]dW(t), \q t\in[0,T],\\
 X(0) &= x
\end{aligned}\right.$$
can be decomposed into
$$ X(\cd)=X_0(\cd)+X_1(\cd)+X_2(\cd)=[\cN_{\scT}x](\cd)+[\cL_{1,\scT}u_1](\cd)+[\cL_{2,\scT}u_2](\cd), $$
and in particular,
$$ X(T) =X_0(T)+X_1(T)+X_2(T) =\wh\cN_{\scT}x+\wh\cL_{1,\scT}u_1+\wh\cL_{2,\scT}u_2. $$
Denote by $\cA^*$ the adjoint operator of a linear operator $\cA$. Then the performance functional
$$J_{\scT}^{\scp0}(x;u_1(\cd),u_2(\cd)) \deq \dbE\int_0^T \Blan\!\!
   \begin{pmatrix}Q   & \!\!S_1^\top & \!\!S_2^\top \\
                  S_1 & \!\!R_{11}   & \!\!R_{12}   \\
                  S_2 & \!\!R_{21}   & \!\!R_{22}   \end{pmatrix}\!\!
   \begin{pmatrix}X(t) \\ u_1(t) \\ u_2(t)\end{pmatrix}\!,\!
   \begin{pmatrix}X(t) \\ u_1(t) \\ u_2(t)\end{pmatrix}\!\!\Bran dt $$
can be represented as follows:
\begin{equation}\label{JT:rep}
J_{\scT}^{\scp0}(x;u_1(\cd),u_2(\cd)) = \lan\cM_{\scT}u,u\ran + 2\lan\cK_{\scT}x,u\ran + \lan\cO_{\scT}x,x\ran,
\end{equation}
where
\begin{equation}\label{def-M}
\cM_{\scT}\deq \begin{pmatrix}\cM_{11,\scT}&\cM_{12,\scT} \\ \cM_{21,\scT}&\cM_{22,\scT}\end{pmatrix}, \q
\cK_{\scT}\deq \begin{pmatrix}\cK_{1,\scT} \\ \cK_{2,\scT}\end{pmatrix}, \q
\cO_{\scT}\deq \cN_{\scT}^*Q\cN_{\scT}, \q
    u(\cd)\deq \begin{pmatrix}u_1(\cd) \\ u_2(\cd)\end{pmatrix},
\end{equation}
with
\begin{align}
\cM_{ij,\scT}
&\deq R_{ij}+S_i\cL_{j,\scT}+\cL_{j,\scT}^*S_i^\top+\cL_{i,\scT}^*Q\cL_{j,\scT}; \q i,j=1,2, \label{def-Mij}\\
\cK_{i,\scT}
&\deq \cL_{i,\scT}^*Q\cN_{\scT}+S_i\cN_{\scT}; \q i=1,2. \label{def-K}
\end{align}
Note that the linear operators
$$\cM_{\scT}:\sU[0,\i)\to\sU[0,\i),
\q\cK_{\scT}:\dbR^n\to\sU[0,\i),
\q\cO_{\scT}\in\dbS^n $$
are all bounded uniformly in $T$, and $\cM_{\scT}$ is self-adjoint.

\ms

Similar to the previous discussion, replacing the interval $[0,T]$ by $[0,\i)$,
we can derive a similar operator representation for the performance functional \rf{cost:IH}:
\begin{equation}\label{J-infty:rep}
J_{\scp\i}^{\scp0}(x;u_1(\cd),u_2(\cd)) = \lan\cM u,u\ran + 2\lan\cK x,u\ran + \lan\cO x,x\ran,
\end{equation}
where $\cO\in\dbS^n$, and the linear operators
$$\cM\deq\begin{pmatrix}\cM_{11}&\cM_{12} \\ \cM_{21}&\cM_{22}\end{pmatrix}: \sU[0,\i)\to\sU[0,\i),\q
  \cK\deq\begin{pmatrix}\cK_1 \\ \cK_2\end{pmatrix}: \dbR^n\to\sU[0,\i) $$
are bounded with $\cM$ being self-adjoint.

\ms

Observe that condition (A1) is equivalent to the uniform positivity of
$\cM_{11,\scT}$ and $-\cM_{22,\scT}$, meaning there exists a constant $\d>0$,
independent of $T$, such that
$$ \lan\cM_{11,\scT}u_1,u_1\ran\ges \d\|u_1(\cd)\|^2,
\q \lan\cM_{22,\scT}u_2,u_2\ran\les-\d\|u_2(\cd)\|^2 $$
for all $u_i(\cd)\in\sU_i[0,T]$; $i=1,2$. Furthermore, condition (A1) implies that
$$\lan\cM_{11}u_1,u_1\ran\ges \d\|u_1(\cd)\|^2,
\q\lan\cM_{22}u_2,u_2\ran\les-\d\|u_2(\cd)\|^2 $$
for all $u_i(\cd)\in\sU_i[0,\i)$; $i=1,2$.

\ms

With the functional representation \rf{J-infty:rep} and the above analysis,
we derive the following uniqueness and existence result.

\begin{theorem}\label{thm:Ginfty-optimality}
Let {\rm(A1)--(A2)} hold. Then for each initial state $x$, Problem (DG)$_{\scp\i}^{\scp0}$
has a unique open-loop saddle strategy. Moreover, a control pair $\bar u(\cd)\equiv
\begin{pmatrix}\bar u_1(\cd) \\ \bar u_2(\cd)\end{pmatrix} \in\sU_1[0,\i)\times\sU_2[0,\i)$
is the saddle strategy if and only if
\begin{equation}\label{FBSDE-kehua}
B^\top\bar Y(t)+D^\top\bar Z(t)+S\bar X(t)+R\bar u(t) = 0, \q\ae~t\in[0,\i),~\as,
\end{equation}
where $(\bar Y(\cd),\bar Z(\cd))$ is the $L^2$-{\it stable adapted solution}\footnote{See the
appendix of \cite{Sun-Yong2020book-a} for the definition and unique existence of such a kind
of solution.} to the following linear BSDE over $[0,\i)$:
\begin{align}\label{BSDE-infty}
d\bar Y(t) = -\lt[A^{\top}\bar Y(t)+C^{\top}\bar Z(t)+Q\bar X(t)+S^{\top}\bar u(t)\rt]dt + \bar Z(t)dW(t),
\end{align}
where $\bar X(\cd)$ is the state process in \rf{state:IH} corresponding to $\bar u(\cd)$.
\end{theorem}

\begin{proof} By definition, a pair $(\bar u_1(\cd),\bar u_2(\cd))$ is an open-loop
saddle strategy if and only if
\begin{equation}\label{Jinifty:saddle}\begin{aligned}
J_{\scp\i}^{\scp0}(x;\bar u_1(\cd),\bar u_2(\cd)+\e v_2(\cd))
\les J_{\scp\i}^{\scp0}(x;\bar u_1(\cd),\bar u_2(\cd))
\les J_{\scp\i}^{\scp0}(x;\bar u_1(\cd)+\e v_1(\cd),\bar u_2(\cd)), \\
\forall\e\in\dbR,~\forall (v_1(\cd),v_2(\cd))\in\sU_1[0,\i)\times\sU_2[0,\i).
\end{aligned}\end{equation}
The operator representation \rf{J-infty:rep} of the performance functional \rf{cost:IH}
can be rewritten as
\begin{align*}
J^{\scp0}_{\scp\i}(x;u_1(\cd),u_2(\cd))
&= \lan\cM_{11}u_1,u_1\ran + \lan\cM_{22}u_2,u_2\ran + 2\lan u_1,\cM_{12}u_2\ran \\
&\hp{=\ } +2\lan u_1,\cK_1x\ran + 2\lan u_2,\cK_2x\ran + \lan\cO x,x\ran.
\end{align*}
In terms of this representation, \rf{Jinifty:saddle} is equivalent to
$$\left\{\begin{aligned}
\e^2\lan\cM_{11}v_1,v_1\ran + 2\e\lan v_1,\cM_{11}\bar u_1 + \cM_{12}\bar u_2+\cK_1x\ran \ges0,
\q\forall\e\in\dbR,~\forall v_1(\cd)\in\sU_1[0,\i); \\
\e^2\lan\cM_{22}v_2,v_2\ran + 2\e\lan v_2,\cM_{22}\bar u_2 + \cM_{21}\bar u_1+\cK_2x\ran \les0,
\q\forall\e\in\dbR,~\forall v_2(\cd)\in\sU_2[0,\i).
\end{aligned}\right.$$
The above is in turn equivalent to
\begin{equation}\label{suanzi-kehua}\left\{\begin{aligned}
\cM_{11}\bar u_1+\cM_{12}\bar u_2+\cK_1x=0, \\
\cM_{22}\bar u_2+\cM_{21}\bar u_1+\cK_2x=0,
\end{aligned}\right.\end{equation}
since for any $v_1(\cd)\in\sU_1[0,\i)$ and any $v_2(\cd)\in\sU_2[0,\i)$,
$$ \lan\cM_{11}v_1,v_1\ran\ges0, \q \lan\cM_{22}v_2,v_2\ran\les0. $$
Note that the operator $\cM$ is invertible, with the inverse $\cM^{-1}$ given by
\begin{align*}
\cM^{-1}=\begin{pmatrix}
\cM_{11}^{-1}+[\cM_{11}^{-1}\cM_{12}]\F^{-1}[\cM_{11}^{-1}\cM_{12}]^* & -[\cM_{11}^{-1}\cM_{12}]\F^{-1} \\
-\F^{-1}[\cM_{11}^{-1}\cM_{12}]^* & \F^{-1}\end{pmatrix},
\end{align*}
where $\F\deq \cM_{22}-\cM_{21}\cM_{11}^{-1}\cM_{12}$ is a negative (and hence invertible) operator.
This shows that Problem (DG)$_{\scp\i}^{\scp0}$ has a unique open-loop saddle strategy, given by
$$\begin{pmatrix}\bar u_1(\cd) \\ \bar u_2(\cd)\end{pmatrix}
= -\begin{pmatrix}\cM_{11}&\cM_{12} \\ \cM_{21}&\cM_{22}\end{pmatrix}^{-1}
   \begin{pmatrix}\cK_1x \\ \cK_2x \end{pmatrix}. $$
If we substitute the integral form of the performance functional \rf{cost:IH} instead of
the operator representation \rf{J-infty:rep} into \rf{Jinifty:saddle}, then a straightforward
computation reveals that the criterion \rf{suanzi-kehua} transforms into \rf{FBSDE-kehua}.
\end{proof}

\begin{remark}\label{remark:EY2go0}
Suppose that the control pair $\bar u(\cd)$ in \autoref{thm:Ginfty-optimality}
is an open-loop saddle strategy for the initial state $x$. Since by \autoref{prop:SDE-bound},
$$ \f(\cd)\deq Q\bar X(\cd)+S^{\top}\bar u(\cd) \in L_{\dbF}^2(0,\i;\dbR^n), $$
we conclude from (A.2.7) in the proof of \cite[Proposition A.2.3]{Sun-Yong2020book-a}
that the $L^2$-stable adapted solution of \rf{BSDE-infty} has the following property:
\begin{align}\label{EY2go0}
\lim_{t\to\i}\dbE|\bar Y(t)|^2 = 0.
\end{align}
\end{remark}

The following result further establishes a convergence of Problem (DG)$_{\scT}^{\scp0}$
to Problem (DG)$_{\scp\i}^{\scp0}$ in a suitable sense.
Such a result essentially implies the convergence of the solution to the differential
Riccati equation \rf{Ric:game}.

\begin{proposition}\label{prop:uT-go-u}
Let {\rm(A1)--(A2)} hold. For a given initial state $x$, let $(\bar u_{1,\scT}(\cd),\bar u_{2,\scT}(\cd))$
and $(\bar u_1(\cd),\bar u_2(\cd))$ be the open-loop saddle strategies of Problem (DG)$_{\scT}^{\scp0}$
and Problem (DG)$_{\scp\i}^{\scp0}$, respectively. Then
\begin{equation}\label{uT-go-u}
\lim_{T\to\i}\dbE\int_0^T\[|\bar u_{1,\scT}(t)-\bar u_1(t)|^2+|\bar u_{2,\scT}(t)-\bar u_2(t)|^2\]dt=0.
\end{equation}
Consequently, for the corresponding state processes $\bX(\cd)$ and $\bar X(\cd)$, we have
\begin{equation}\label{XT-go-X}
\lim_{T\to\i}\dbE\int_0^T|\bar X_{\scT}(t)-\bar X(t)|^2dt=0.
\end{equation}
\end{proposition}

\begin{proof}
Let $\bar u_{\scT}(\cd)\deq\begin{pmatrix}\bar u_{1,\scT}(\cd) \\ \bar u_{2,\scT}(\cd)\end{pmatrix}$
and $\bar u(\cd)\deq\begin{pmatrix}\bar u_1(\cd) \\ \bar u_2(\cd)\end{pmatrix}$.
By \autoref{thm:uT-rep}(ii), the adapted solution $(\bX(\cd),\bY(\cd),\bZ(\cd))$
of the forward-backward SDE (FBSDE, for short)
$$\left\{\begin{aligned}
d\bX(t) & =  [A\bX(t)+B\bar u_{\scT}(t)]dt + [C\bX(t)+D\bar u_{\scT}(t)] dW(t), \\
d\bY(t) & = -[A^{\top}\bY(t)+C^{\top}\bZ(t)+Q\bX(t)+S^{\top}\bar u_{\scT}(t)]dt + \bZ(t)dW(t), \\
 \bX(0) & = x, \q \bY(T)=0
\end{aligned}\right.$$
satisfies
$$ B^\top\bY(t)+D^\top\bZ(t)+S\bX(t)+R\bar u_{\scT}(t) = 0, \q\ae~t\in[0,T],~\as$$
On the other hand, by \autoref{thm:Ginfty-optimality},
$$ B^\top\bar Y(t)+D^\top\bar Z(t)+S\bar X(t)+R\bar u(t) = 0, \q\ae~t\in[0,\i),~\as$$
where $\bar X(\cd)$ is the solution of
$$\left\{\begin{aligned}
d\bar X(t) & =[A\bar X(t)+B\bar u(t)]dt + [C\bar X(t)+D\bar u(t)] dW(t), \q t\in[0,\i), \\
 \bar X(0) & =x,
\end{aligned}\right.$$
and $(\bar Y(\cd),\bar Z(\cd))$ is the $L^2$-stable adapted solution of
$$ d\bar Y(t) = -[A^{\top}\bar Y(t)+C^{\top}\bar Z(t)+Q\bar X(t)+S^{\top}\bar u(t)]dt
 + \bar Z(t)dW(t), \q t\in[0,\i).$$
Define for $t\in[0,T]$,
\begin{align*}
& \hu_{\scT}(t)= \begin{pmatrix}\hu_{1,\scT}(t) \\ \hu_{2,\scT}(t)\end{pmatrix}
            \deq \bar u(t)-\bar u_{\scT}(t)
               = \begin{pmatrix}\bar u_1(t)-\bar u_{1,\scT}(t) \\ \bar u_2(t)-\bar u_{2,\scT}(t)\end{pmatrix}, \\
& \wh X_{\scT}(t)  \deq \bar X(t)-\bar X_{\scT}(t), \q
 {\wh Y}_{\scT}(t) \deq \bar Y(t)-\bar Y_{\scT}(t), \q
 {\wh Z}_{\scT}(t) \deq \bar Z(t)-\bar Z_{\scT}(t).
\end{align*}
Then on the interval $[0,T]$,
\begin{equation}\label{FBSDE:hX}\left\{\begin{aligned}
& d{\wh X}_{\scT}(t) = [A{\wh X}_{\scT}(t)+B\hu_{\scT}(t)]dt+[C{\wh X}_{\scT}(t)+D\hu_{\scT}(t)]dW(t),\\
& d{\wh Y}_{\scT}(t) = -[A^{\top}{\wh Y}_{\scT}(t)+C^{\top}{\wh Z}_{\scT}(t)+Q{\wh X}_{\scT}(t)
                       + S^{\top}\hu_{\scT}(t)]dt + {\wh Z}_{\scT}(t)dW(t), \\
&  {\wh X}_{\scT}(0) = 0, \q{\wh Y}_{\scT}(T) = \bar Y(T),
\end{aligned}\right.\end{equation}
and the following holds:
\begin{equation}\label{pingwei:hX}
B^\top{\wh Y}_{\scT}(t)+D^\top{\wh Z}_{\scT}(t)+S{\wh X}_{\scT}(t)+R\hu_{\scT}(t)=0, \q\ae~t\in[0,T],~\as
\end{equation}
Again, by \autoref{thm:uT-rep}(ii), we see from \rf{FBSDE:hX}--\rf{pingwei:hX}
that $\hu_{\scT}(\cd)$ is the (unique) open-loop saddle strategy for the initial state $x=0$ of the zero-sum stochastic LQ differential game with the state equation
$$\left\{\begin{aligned}
dX(t) &= [AX(t)+B_1u_1(t)+B_2u_2(t)]dt + [CX(t)+D_1u_1(t)+D_2u_2(t)]dW(t), \q t\in[0,T],\\
 X(0) &= x
\end{aligned}\right.$$
and the performance functional
\begin{align*}
\wh J_{\scT}(x;u_1(\cd),u_2(\cd)) &\deq \dbE\bigg[2\lan\bar Y(T),X(T)\ran + \!\int_0^T\!\Blan\!\!
   \begin{pmatrix}Q   & \!\!S_1^\top & \!\!S_2^\top \\
                  S_1 & \!\!R_{11}   & \!\!R_{12}   \\
                  S_2 & \!\!R_{21}   & \!\!R_{22}   \end{pmatrix}\!\!
   \begin{pmatrix}X(t) \\ u_1(t) \\ u_2(t)\end{pmatrix}\!,\!
   \begin{pmatrix}X(t) \\ u_1(t) \\ u_2(t)\end{pmatrix}\!\!\Bran dt\bigg].
\end{align*}
In terms of the operator $\cM_{\scT}$ defined in \rf{def-M} and the operator
$$ \wh\cL_{\scT} \deq (\wh\cL_{1,\scT},\wh\cL_{2,\scT}), $$
where $\wh\cL_{1,\scT}$ and $\wh\cL_{2,\scT}$ are defined in \rf{def:cL},
we can represent $\wh J_{\scT}(0;u_1(\cd),u_2(\cd))$ as
$$\wh J_{\scT}(0;u_1(\cd),u_2(\cd))
= \lan\cM_{\scT} u,u\ran + 2\lan \wh\cL_{\scT}^*\bar Y(T),u\ran. $$
Form this representation it is easily seen that the saddle strategy $\hu_{\scT}(\cd)$ is given by
$$ \hu_{\scT}(\cd) = -\cM_{\scT}^{-1}\wh\cL_{\scT}^*\bar Y(T). $$
Observe that
$$
\cM_{\scT}^{-1}=\begin{pmatrix}
\cM_{11,\scT}^{-1}+[\cM_{11,\scT}^{-1}\cM_{12,\scT}]\F_{\scT}^{-1}[\cM_{11,\scT}^{-1}
\cM_{12,\scT}]^*
&-[\cM_{11,\scT}^{-1}\cM_{12,\scT}]\F_{\scT}^{-1}\\
-\F_{\scT}^{-1}[\cM_{11,\scT}^{-1}\cM_{12,\scT}]^*
&\F_{\scT}^{-1}\end{pmatrix}, $$
where $\F_{\scT}\deq \cM_{22,\scT}-\cM_{21,\scT}\cM_{11,\scT}^{-1}\cM_{12,\scT}$.
By (A1), there exists a constant $\d>0$, independent of $T$, such that the self-adjoint
operators $\cM_{11,\scT}$ and $\cM_{22,\scT}$ satisfy
\begin{equation}\label{M11+M22}
\cM_{11,\scT}\ges\d\,\cI,\q\cM_{22,\scT}\les-\d\,\cI,
\end{equation}
where $\cI$ denotes the identity operator. Clearly, \rf{M11+M22} implies that
$\cM_{11,\scT}^{-1}$, $\F_{\scT}^{-1}$, and hence $\cM_{\scT}^{-1}$ are bounded
linear operators, and that the operator norm
$$\|\cM_{\scT}^{-1}\|\les K,$$
for some constant $K>0$ that is independent of $T$.
Recall that the operator $\wh\cL_{\scT}$ is also bounded uniformly in $T$.
Thus, the following holds for some constant $K>0$ independent of $T$:
$$\dbE\int_0^T\[|\bar u_{1,\scT}(t)-\bar u_1(t)|^2+|\bar u_{2,\scT}(t)-\bar u_2(t)|^2\]dt
=\|\hu_{\scT}(\cd)\|^2\les K\dbE|\bar Y(T)|^2. $$
The desired \rf{uT-go-u} then follows from \autoref{remark:EY2go0}.
Moreover, as ${\wh X}_{\scT}(\cd) \deq \bar X(\cd)-\bX(\cd)$ satisfies
$$\left\{\begin{aligned}
d{\wh X}_{\scT}(t) & =  [A{\wh X}_{\scT}(t)+B\hu_{\scT}(t)]dt + [C{\wh X}_{\scT}(t)+D\hu_{\scT}(t)] dW(t), \\
 {\wh X}_{\scT}(0) & = 0,
\end{aligned}\right.$$
by applying \autoref{prop:SDE-bound} along with \rf{uT-go-u}, we obtain \rf{XT-go-X}.
\end{proof}

\section{Riccati Equations}\label{Sec:Ric}

Based on the results of the above section, in this section, we explore the asymptotic behavior
of the solution $P_{\scT}(\cd)$ to the differential Riccati equation \rf{Ric:game} as $T\to\i$.
Additionally, as a byproduct, we establish a closed-loop representation for the open-loop
saddle strategy of Problem (DG)$_{\scp\i}^{\scp0}$.

\ms

First, let us consider the following differential Riccati equation over $[0,T]$:
\begin{equation}\label{Ric:1}\left\{\begin{aligned}
& \dot P_{\scp1T} + P_{\scp1T}A + A^\top P_{\scp1T} + C^\top P_{\scp1T}C + Q \\
& \hp{\dot P_{\scp1T}} -(P_{\scp1T}B_1\!+\!C^\top\!P_{\scp1T}D_1\!+\!S_1^\top)
  (R_{11}\!+\!D_1^\top\!P_{\scp1T}D_1)^{-1}(B_1^\top\!P_{\scp1T}\!+\!D_1^\top\!P_{\scp1T}C\!+\!S_1)=0, \\
& P_{\scp1T}(T)=0.
\end{aligned}\right.\end{equation}
By (A1), there exists a constant $\d>0$ such that
$$J_{\scT}^{\scp0}(0;u_1(\cd),0) \ges \d\,\dbE\int_0^T|u_1(t)|^2dt,  \q\forall u_1(\cd)\in\sU_1[0,T]. $$
According to \cite[Theorem 2.5.6]{Sun-Yong2020book-a}, \rf{Ric:1} admits a unique solution
$P_{\scp1T}(\cd)\in C([0,T];\dbS^n)$ satisfying
\begin{equation}\label{R11+D1P1D1>>0}
  R_{11}+D_1^\top P_{\scp1T}(t)D_1\ges \a I_n, \q\forall t\in[0,T]
\end{equation}
for some constant $\a>0$. Note that if we let
$$ \Th_{\scp1T}(\cd) \deq
-(R_{11}+D_1^\top P_{\scp1T}D_1)^{-1}(B_1^\top P_{\scp1T}+D_1^\top P_{\scp1T}C+S_1)(\cd), $$
the equation \rf{Ric:1} can be rewritten as
$$\left\{\begin{aligned}
& \dot P_{\scp1T} + P_{\scp1T}(A+B_1\Th_{\scp1T}) + (A+B_1\Th_{\scp1T})^\top P_{\scp1T}
  + (C+D_1\Th_{\scp1T})^\top P_{\scp1T}(C+D_1\Th_{\scp1T}) \\
& \hp{\dot P_{\scp1T}} +\Th_{\scp1T}^\top R_{11}\Th_{\scp1T} + S_1^\top\Th_{\scp1T}
  + \Th_{\scp1T}^\top S_1  + Q =0, \\
& P_{\scp1T}(T)=0.
\end{aligned}\right.$$
Then by \cite[Proposition 2.5.5]{Sun-Yong2020book-a}, the constant $\a$ in \rf{R11+D1P1D1>>0}
can be chosen to be the same as the $\d$ in (A1), that is, the following holds:
\begin{equation}\label{R11+D1P1D1>delta}
  R_{11}+D_1^\top P_{\scp1T}(t)D_1\ges \d I_n, \q\forall t\in[0,T].
\end{equation}
Likewise, the differential Riccati equation
\begin{equation}\label{Ric:2}\left\{\begin{aligned}
&\dot P_{\scp2T} + P_{\scp2T}A + A^\top P_{\scp2T} + C^\top P_{\scp2T}C + Q \\
&\hp{\dot P_{\scp2T}} -(P_{\scp2T}B_2\!+\!C^\top\!P_{\scp2T}D_2\!+\!S_2^\top)
(R_{22}\!+\!D_2^\top\!P_{\scp2T}D_2)^{-1}(B_2^\top\!P_{\scp2T}\!+\!D_2^\top\!P_{\scp2T}C\!+\!S_2)=0, \\
& P_{\scp2T}(T)=0
\end{aligned}\right.\end{equation}
admits a unique solution $P_{\scp2T}(\cd)\in C([0,T];\dbS^n)$ satisfying
\begin{equation}\label{R22+D2P2D2<delta}
  R_{22}+D_2^\top P_{\scp2T}(t)D_2 \les -\d I_n, \q\forall t\in[0,T].
\end{equation}
Taking into account \cite[Proposition 4.7]{Sun2021}, we obtain the following result.

\begin{proposition}\label{prpo:comparison}
Let {\rm(A1)} hold. Let $P_{\scT}(\cd)$, $P_{\scp1T}(\cd)$, and $P_{\scp2T}(\cd)$ be the solutions
to \rf{Ric:game}, \rf{Ric:1}, and \rf{Ric:2}, respectively. Then
$$ P_{\scp1T}(t) \les P_{\scT}(t) \les P_{\scp2T}(t), \q\forall t\in[0,T]. $$
Consequently,
\begin{align}\label{uni-bound-PT}
R_{11}+D_1^\top P_{\scT}(t)D_1\ges \d I_n,
\q R_{22}+D_2^\top P_{\scT}(t)D_2 \les -\d I_n,
\q \forall t\in[0,T].
\end{align}
\end{proposition}

To further discuss the properties of $P_{\scT}(\cd)$, let us revisit the concept of $L^2$-stabilizability.
Recall the following notation:
$$ m \deq m_1 + m_2, \q B \deq (B_1, B_2), \q D \deq (D_1, D_2). $$
Denote by $[A,C;B,D]$ the following controlled linear system:
\begin{equation}\label{[ACBD]}
dX(t) = [AX(t)+Bu(t)]dt + [CX(t)+Du(t)]dW(t), \q t\ges0.
\end{equation}

\begin{definition}
System $[A,C;B,D]$ is called {\it $L^2$-stabilizable} if there exists a matrix $\Th\in\dbR^{m\times n}$
such that the (closed-loop) system $[A+B\Th,C+D\Th]$ is $L^2$-stable.
In this case, $\Th$ is called a {\it stabilizer} of $[A,C;B,D]$.
\end{definition}

Now, we present the following result, which establishes the convergence of $\lim_{T\to\i}P_{\scT}(t)$
and provides some properties for the limit matrix $P$.

\begin{theorem}\label{thm:PTgoP}
Let {\rm(A1)--(A2)} hold. Let $P_{\scT}(\cd)$ be the unique strongly regular solution
to the differential Riccati equation \rf{Ric:game}. Then the limit
$$ P\deq\lim_{T\to\i}P_{\scT}(t) $$
exists, is independent of $t$, and has the following properties:

\ms

{\rm(i)} $R_{11} + D_1^\top PD_1>0$, $R_{22} + D_2^\top PD_2 <0$.

\ms

{\rm(ii)} $P$ satisfies the following algebraic Riccati equation (ARE, for short):
\begin{equation}\label{ARE:game}\begin{aligned}
      & PA + A^\top P + C^\top PC + Q\\
      &\hp{PA} -(PB+C^\top PD+S^\top)(R+D^\top PD)^{-1}(B^\top P+D^\top PC+S)=0.
\end{aligned}\end{equation}

{\rm(iii)} The matrix
\begin{equation}\label{def:Th}
\Th \deq -(R+D^\top PD)^{-1}(B^\top P+D^\top PC+S)
\end{equation}
is a stabilizer of the system $[A,C;B,D]$.
\end{theorem}

\begin{proof}
According to \cite[Theorem 4.4]{Sun2021}, the value function $V_{\scT}^{\scp0}(\cd)$
of Problem (DG)$_{\scT}^{\scp0}$ is given by
$$V_{\scT}^{\scp0}(x) = \lan P_{\scT}(0)x,x\ran, \q\forall x\in\dbR^n. $$
Denote by $V_{\scp\i}^{\scp0}(\cd)$ the value function of Problem (DG)$_{\scp\i}^{\scp0}$, and let
$$\bar u_{\scT}(\cd) \deq \begin{pmatrix}\bar u_{1,\scT}(\cd) \\ \bar u_{2,\scT}(\cd)\end{pmatrix}
\q\text{and}\q
\bar u(\cd) \deq \begin{pmatrix}\bar u_1(\cd) \\ \bar u_2(\cd)\end{pmatrix}$$
be the open-loop saddle strategies of Problem (DG)$_{\scT}^{\scp0}$ and Problem (DG)$_{\scp\i}^{\scp0}$
for the initial state $x$, respectively. Then, by \autoref{prop:uT-go-u},
\begin{align*}
& \lim_{T\to\i}\lan P_{\scT}(0)x,x\ran = \lim_{T\to\i}V_{\scT}^{\scp0}(x)
  =\lim_{T\to\i}J_{\scT}^{\scp0}(x;\bar u_{1,\scT}(\cd),\bar u_{2,\scT}(\cd)) \\
&\q= \lim_{T\to\i}\dbE\int_0^T\Blan\!\!
     \begin{pmatrix*}[l]Q & \!\!S^\top \\ S & \!\!R \end{pmatrix*}\!\!
     \begin{pmatrix}\bX(t) \\ \bar u_{\scT}(t) \end{pmatrix}\!,\!
     \begin{pmatrix}\bX(t) \\ \bar u_{\scT}(t) \end{pmatrix}\!\!\Bran dt \\
&\q= \dbE\int_0^\i\Blan\!\!
     \begin{pmatrix*}[l]Q & \!\!S^\top \\ S & \!\!R \end{pmatrix*}\!\!
     \begin{pmatrix}\bar X(t) \\ \bar u(t) \end{pmatrix}\!,\!
     \begin{pmatrix}\bar X(t) \\ \bar u(t) \end{pmatrix}\!\!\Bran dt
     =J^{\scp0}_{\scp\i}(x;\bar u_1(\cd),\bar u_2(\cd))=V_{\scp\i}^{\scp0}(x).
\end{align*}
Since the initial state $x\in\dbR^n$ is arbitrary, we conclude that the limit
$P\deq\lim_{T\to\i}P_{\scT}(0)$ exists.
Observe that for any $0\les t\les T$ and any $0\les s\les T-t$,
\begin{equation}\label{PT:group-prop}
P_{\scT}(t+s)=P_{\scT-t}(s).
\end{equation}
In particular, taking $s=0$, we get
\begin{equation}\label{PT:limit}
\lim_{T\to\i}P_{\scT}(t)=\lim_{T\to\i}P_{\scT-t}(0)=P.
\end{equation}
The property (i) follows directly from \rf{uni-bound-PT}. To prove (ii), let
\begin{align*}
\L_{\scT}(t) &\deq P_{\scT}(t)A + A^\top P_{\scT}(t) + C^\top P_{\scT}(t)C + Q \\
&\hp{=\ } -[P_{\scT}(t)B+C^\top P_{\scT}(t)D+S^\top][R+D^\top P_{\scT}(t)D]^{-1}
           [B^\top P_{\scT}(t)+D^\top P_{\scT}(t)C+S].
\end{align*}
From \rf{PT:group-prop} and ODE \rf{Ric:game}, we have
$$ P_{\scT}(1)-P_{\scT}(0)=-\int_0^1\L_{\scT}(t)dt
= -\int_0^1 \L_{\scT-t}(0)dt = \int^T_{T-1}\L_{t}(0)dt. $$
Letting $T\to\i$ in the above, we obtain
\begin{align*}
0=\lim_{t\to\i}\L_{t}(0) &= PA+A^\top P+C^\top PC+Q \\
&\hp{=\ } -(PB+C^\top PD+S^\top)(R+D^\top PD)^{-1}(B^\top P+D^\top PC+S).
\end{align*}
It remains to prove (iii). Applying \autoref{thm:uT-rep} to the case of
Problem (DG)$_{\scT}^{\scp0}$, we see that the open-loop saddle strategy
$(\bar u_{1,\scT}(\cd),\bar u_{2,\scT}(\cd))$ of Problem (DG)$_{\scT}^{\scp0}$
has the following closed-loop representation:
$$ \bar u_{\scT}(t) \deq \begin{pmatrix}\bar u_{\sc1,\scT}(t) \\ \bar u_{\sc2,\scT}(t)\end{pmatrix}
= \bar\Th_{\scT}(t)\bX(t), \q t\in[0,T], $$
where $\bar\Th_{\scT}(\cd)$ is defined by \rf{def:barTh} and $\bX(\cd)$ is the solution of the closed-loop system
$$\left\{\begin{aligned}
d\bX(t) &= [A+B\bar\Th_{\scT}(t)]\bX(t)dt + [C+D\bar\Th_{\scT}(t)]\bX(t) dW(t), \q t\in[0,T], \\
 \bX(0) &= x.
\end{aligned}\right.$$
Note that by \rf{PT:group-prop} and \rf{PT:limit}, the following holds for some constant $K>0$:
\begin{equation}\label{PT:uni-bound}
|P_{\scT}(t)|=|P_{\scT-t}(0)|\les K,\q\forall0\les t\les T,
\end{equation}
which in turn implies that
\begin{equation}\label{ThT:uni-bound}
|\bar\Th_{\scT}(t)|\les K,\q\forall0\les t\les T,
\end{equation}
for some possibly different constant $K>0$. Also, note that
\begin{equation}\label{ThT:lim}
\lim_{T\to\i}\bar\Th_{\scT}(t) = \Th = -(R+D^\top PD)^{-1}(B^\top P+D^\top PC+S).
\end{equation}
Let $\bar X(\cd)$ be the state process corresponding to the saddle strategy
$\bar u(\cd)$ of Problem (DG)$_{\scp\i}^{\scp0}$.
Then by \autoref{prop:uT-go-u}, \rf{ThT:uni-bound}, and \rf{ThT:lim},
\begin{align*}
& \dbE\int_0^\i|\bar u(t)-\Th\bar X(t)|^2dt = \lim_{T\to\i}\dbE\int_0^T|\bar u(t)-\Th\bar X(t)|^2dt \\
&\q\les 2\lim_{T\to\i}\lt[\dbE\int_0^T|\bar u(t)-\bar u_{\scT}(t)|^2dt
        +\dbE\int_0^T|\bar u_{\scT}(t)-\Th\bar X(t)|^2dt\rt] \\
&\q= 2\lim_{T\to\i}\dbE\int_0^T|\bar\Th_{\scT}(t)\bX(t)-\Th\bar X(t)|^2dt \\
&\q\les 4\lim_{T\to\i}\lt[\dbE\int_0^T|\bar\Th_{\scT}(t)|^2|\bX(t)-\bar X(t)|^2dt
        +\dbE\int_0^T|\bar\Th_{\scT}(t)-\Th|^2|\bar X(t)|^2dt\rt]=0.
\end{align*}
This implies that $\bar u(\cd)=\Th\bar X(\cd)$ and hence
$$\left\{\begin{aligned}
d\bar X(t) &= (A+B\Th)\bar X(t)dt + (C+D\Th)\bar X(t)dW(t), \q t\ges 0, \\
 \bar X(0) &= x.
\end{aligned}\right.$$
Since we already know that $\bar u(\cd)\in\sU[0,\i)$
(implying $\bar X(\cd)\in L_{\dbF}^2(0,\i;\dbR^n)$ by \autoref{prop:SDE-bound}),
$\Th$ is, by definition, a stabilizer of the system $[A,C;B,D]$.
\end{proof}

From the above proof, we also obtain the following corollary.

\begin{corollary}\label{crllry:unibound:R+DPD}
Let {\rm(A1)--(A2)} hold. Then there exists a constant $K>0$ such that
$$|[R+D^\top P_{\scT}(t)D]^{-1}|\les K,\q\forall0\les t\les T<\i.$$
\end{corollary}

\begin{proof}
For notational simplicity, let us fix $0\les t\les T<\i$ and set
$$M_{ij}(t)\deq R_{ij}+D_i^\top P_{\scT}(t)D_j;\q i,j=1,2.$$
Form \rf{uni-bound-PT} and \rf{PT:uni-bound}, we see that for some constants $\d,\rho>0$
that are independent of $t$ and $T$,
$$ M_{11} \ges \d I, \q M_{22} \les -\d I, \q |M_{12}|=|M_{21}| \les \rho. $$
It is easy to verify that the inverse of
$$R+D^\top P_{\scT}(t)D
=\begin{pmatrix}
R_{11}+ D_1^\top P_{\scT}(t)D_1 & R_{12} + D_1^\top P_{\scT}(t)D_2 \\[1mm]
R_{21}+ D_2^\top P_{\scT}(t)D_1 & R_{22} + D_2^\top P_{\scT}(t)D_2
\end{pmatrix}
=\begin{pmatrix}
M_{11}(t)&M_{12}(t)\\[1mm]
M_{21}(t)&M_{22}(t)
\end{pmatrix}$$
is given by
\begin{align}\label{exp:R+DPD}
[R+D^\top P_{\scT}(t)D]^{-1}
=\begin{pmatrix}
M_{11}^{-1}+(M_{11}^{-1}M_{12})N^{-1}(M_{11}^{-1}M_{12})^\top & -(M_{11}^{-1}M_{12})N^{-1} \\
-N^{-1}(M_{11}^{-1}M_{12})^\top & N^{-1} \end{pmatrix},
\end{align}
where
$$ N \deq M_{22}-M_{21}M_{11}^{-1}M_{12} \les -\d I. $$
Note that
$$ |M_{11}^{-1}| \les \sqrt{n}\d^{-1}, \q |N^{-1}|\les \sqrt{n}\d^{-1}, \q |M_{12}|=|M_{21}| \les \rho. $$
The desired result then follows from \rf{exp:R+DPD}.
\end{proof}

\autoref{thm:PTgoP} establishes the existence of a solution to the ARE \rf{ARE:game}.
The following result further confirms the uniqueness of this solution.

\begin{theorem}\label{thm:ARE-uniqueness}
Let {\rm(A1)--(A2)} hold. The ARE \rf{ARE:game} has unique solution $P\in\dbS^n$
satisfying the following conditions:
\begin{enumerate}[\rm(i)]
\item $R_{11}+D_1^\top PD_1>0$, $R_{22} + D_2^\top PD_2<0$;
\item $-(R+D^\top PD)^{-1}(B^\top P+D^\top PC+S)$ is a stabilizer of the system $[A,C;B,D]$.
\end{enumerate}
\end{theorem}

\begin{proof}
Suppose that $P$ is a solution of \rf{ARE:game} such that the conditions stated in (i) and (ii) hold. According to \autoref{thm:Ginfty-optimality}, Problem (DG)$_{\scp\i}^{\scp0}$ has a unique open-loop saddle strategy for every initial state $x$. If we can prove that the value function $V_{\scp\i}^{\scp0}(\cd)$ of Problem (DG)$_{\scp\i}^{\scp0}$ takes the form
\begin{align}\label{V0:rep}
V_{\scp\i}^{\scp0}(x) = \lan Px,x\ran, \q\forall x\in\dbR^n,
\end{align}
then uniqueness follows immediately. To this end, let $\Th$ be defined by \rf{def:Th}
and consider the following SDE:
\begin{equation*}\left\{\begin{aligned}
dX(t) &= (A+B\Th)X(t)dt + (C+D\Th)X(t)dW(t), \q t\ges 0, \\
 X(0) &= x.
\end{aligned}\right.\end{equation*}
Since $\Th$ is a stabilizer of system $[A,C;B,D]$, we have
$$ X(\cd)\in L_{\dbF}^2(0,\i;\dbR^n),\q u(\cd)\deq\Th X(\cd)\in \sU[0,\i).$$
Define
$$Y(t)\deq PX(t), \q Z(t)\deq P(C+D\Th)X(t), \q t\ges0. $$
It is straightforward to verify that
$$ B^\top Y(t)+D^\top Z(t)+S X(t)+R u(t) = 0, \q\ae~t\in[0,\i),~\as $$
Furthermore, noting that
\begin{align*}
PA + A^\top P + C^\top PC + Q +(PB+C^\top PD+S^\top)\Th=0,
\end{align*}
we have
\begin{align*}
dY(t) &= PdX(t) = P[AX(t)+Bu(t)]dt + P[CX(t)+Du(t)]dW(t) \\
&=-\lt[A^\top P + C^\top PC + Q+(PB+C^\top PD+S^\top)\Th-PB\Th\rt]X(t)dt + Z(t)dW(t), \\
&=-\lt[A^\top Y(t) + C^\top Z(t) + QX(t) + S^\top u(t)\rt]dt + Z(t)dW(t).
\end{align*}
Using \autoref{thm:Ginfty-optimality} again, we conclude that $u(\cd)\deq \Th X(\cd)$
is the unique open-loop saddle strategy of Problem (DG)$_{\scp\i}^{\scp0}$ for $x$.
Substituting $u(\cd)\deq \Th X(\cd)$ into the performance functional yields
\begin{align}\label{1-5:1}
J_{\scp\i}^{\scp0}(x;u_1(\cd),u_2(\cd))
= \dbE\int_0^\i\lan(Q+S^\top\Th+\Th^\top S +\Th^\top R\Th)X(t),X(t)\ran dt.
\end{align}
On the other hand, by It\^{o}'s rule,
\begin{align*}
& \dbE\lan PX(t),X(t)\ran-\lan Px,x\ran \\
&\q= \dbE\int_0^t\lan[P(A+B\Th)+(A+B\Th)^\top P+(C+D\Th)^\top P(C+D\Th)]X(s),X(s)\ran ds.
\end{align*}
Letting $t\to\i$ gives
\begin{align*}
 -\lan Px,x\ran = \dbE\int_0^\i\lan[P(A+B\Th)+(A+B\Th)^\top P+(C+D\Th)^\top P(C+D\Th)]X(t),X(t)\ran dt.
\end{align*}
Adding the above to \rf{1-5:1} and noting that \rf{ARE:game} can be rewritten as
$$ P(A+B\Th)+(A+B\Th)^\top P+(C+D\Th)^\top P(C+D\Th)+Q+S^\top\Th+\Th^\top S +\Th^\top R\Th=0,$$
we obtain $J_{\scp\i}^{\scp0}(x;u_1(\cd),u_2(\cd))-\lan Px,x\ran=0$, which implies \rf{V0:rep}.
\end{proof}

From the proof of \autoref{thm:ARE-uniqueness}, we immediately obtain the following corollary.

\begin{corollary}
Let {\rm(A1)--(A2)} hold. Let $P\in\dbS^n$ be the unique solution of the ARE \rf{ARE:game}
satisfying the conditions in \autoref{thm:ARE-uniqueness}. Then the unique open-loop saddle
strategy of Problem (DG)$_{\scp\i}^{\scp0}$ (for the initial state $x$) is given by
$$ \begin{pmatrix}\bar u_1(t) \\ \bar u_2(t)\end{pmatrix} = \Th\bar X(t), \q t\ges0, $$
where $\Th$ is defined by \rf{def:Th} and $\bar X(\cd)$ is the solution to the closed-loop system
\begin{equation*}\left\{\begin{aligned}
d\bar X(t) &= (A+B\Th)\bar X(t)dt + (C+D\Th)\bar X(t)dW(t), \q t\ges 0, \\
 \bar X(0) &= x.
\end{aligned}\right.\end{equation*}
\end{corollary}

We have shown in \autoref{thm:PTgoP} that the solution $P_{\scT}(\cd)$ to the differential
Riccati equation \rf{Ric:game} converges to some constant matrix $P$ as $T\to\i$.
In the subsequent result, we quantify the rate of this convergence.

\begin{theorem}\label{thm:PT-P}
Let {\rm(A1)--(A2)} hold. Let $P_{\scT}(\cd)$ be the unique strongly regular solution
to the differential Riccati equation \rf{Ric:game}, and let $P$ be as in \autoref{thm:ARE-uniqueness}.
There exist constants $K,\l>0$, independent of $T$, such that
$$ |P_{\scT}(t)-P| \les Ke^{-\l(T-t)}, \q\forall t\in[0,T]. $$
\end{theorem}

\begin{proof}
Set $\Si_{\scT}(t)\deq P-P_{\scT}(t)$. Then a direct computation shows that
\begin{align*}
& \dot\Si_{\scT}(t) +\Si_{\scT}(t)(A+B\Th)+(A+B\Th)^\top \Si_{\scT}(t) + (C+D\Th)^\top \Si_{\scT}(t)(C+D\Th) \\
&\hp{\Si_{\scT}(t)} +[\Th-\Th_{\scT}(t)]^\top[R+D^\top P_{\scT}(t)D][\Th-\Th_{\scT}(t)]=0,
\end{align*}
where $\Th_{\scT}(t)$ and $\Th(t)$ are defined by
$$ \Th_{\scT}(t) \deq -[R+D^\top P_{\scT}(t)D]^{-1}[B^\top P_{\scT}(t)+D^\top P_{\scT}(t)C+S] $$
and
$$ \Th \deq -(R+D^\top PD)^{-1}(B^\top P+D^\top PC+S), $$
respectively. Observe that
\begin{align*}
\Th-\Th_{\scT}(t)
&= -[R+D^\top P_{\scT}(t)D]^{-1}[B^\top\Si_{\scT}(t)+D^\top\Si_{\scT}(t)(C+D\Th)].
\end{align*}
Thus, by \autoref{crllry:unibound:R+DPD},
\begin{align}\label{Lambda:guji}
|[\Th-\Th_{\scT}(t)]^\top[R+D^\top P_{\scT}(t)D][\Th-\Th_{\scT}(t)]|
&\les K_1 |\Si_{\scT}(t)|^2, \q\forall 0\les t\les T<\i
\end{align}
for some constant $K_1>0$ that is independent of $T$.
Let $\F(\cd)$ be the solution to the following matrix SDE:
$$\left\{\begin{aligned}
d\F(t) &= (A+B\Th)\F(t)dt + (C+D\Th)\F(t)dW(t),\q t\ges0, \\
 \F(0) &= I_n.
\end{aligned}\right.$$
By \autoref{thm:ARE-uniqueness}, the system $[A+B\Th,C+D\Th]$ is $L^2$-stable.
Thus, by \autoref{crllry:stability}, there exist constants $K_2,\l>0$ such that
\begin{align}\label{F(s)F(t):guji}
\dbE|\F(s)\F(t)^{-1}|^2 = \dbE|\F(s-t)|^2 \les K_2 e^{-\l(s-t)}, \q\forall s\ges t\ges0.
\end{align}
By \autoref{thm:PTgoP}, we can choose an integer $N>0$ such that
\begin{align}\label{guji}
K_2|\Si_{\scT}(0)| \les \rho\deq {\l\over2K_1K_2}, \q\forall T\ges N.
\end{align}
Now we claim that for any $T\ges 2N$, the inequality
\begin{align}\label{Si:bound}
|\Si_{\scT}(t)| \les Ke^{-\l(T-t)}, \q\forall t\in[0,T]
\end{align}
holds for some constant $K>0$ independent of $T$.
To prove this claim, let $T\ges 2N$ be fixed but arbitrary.
Let $k\ges N$ be the integer such that
$$ N+k \les T < N+k+1. $$
By the proof of \cite[Lemma 7.3, Chapter 6]{Yong-Zhou1999}, we see that for any $t\les k$,
\begin{align}\label{Rep:SiT}
\Si_{\scT}(t)
&= \dbE\bigg\{\int_t^{k}[\F(s)\F(t)^{-1}]^\top[\Th-\Th_{\scT}(s)]^\top
   [R+D^\top P_{\scT}(s)D][\Th-\Th_{\scT}(s)][\F(s)\F(t)^{-1}]ds \nn\\
&\hp{=\ } +[\F(k)\F(t)^{-1}]^\top \Si_{\scT}(k)[\F(k)\F(t)^{-1}]\Big\},
\end{align}
which, together with \rf{PT:group-prop} and \rf{Lambda:guji}--\rf{guji}, implies that
\begin{align*}
|\Si_{\scT}(t)|
&\les |\Si_{\scT}(k)|\cd\dbE|\F(k)\F(t)^{-1}|^2 + K_1\int_t^{k}\dbE|\F(s)\F(t)^{-1}|^2\cd|\Si_{\scT}(s)|^2ds \\
&= |\Si_{\scT-k}(0)|\cd\dbE|\F(k)\F(t)^{-1}|^2 + K_1\int_t^{k}\dbE|\F(s)\F(t)^{-1}|^2\cd|\Si_{\scT}(s)|^2ds \\
&\les \rho e^{-\l(k-t)}+ K_1K_2\int_t^{k}e^{-\l(s-t)}|\Si_{\scT}(s)|^2ds, \q\forall 0\les t\les k.
\end{align*}
Set $K_3\deq K_1K_2$ and
$$ h(t) \deq K_3 e^{\l t}|\Si_{\scT}(k-t)|, \q 0\les t\les k. $$
Then for any $0\les t\les k$,
$$ h(t) \les K_3\rho + \int_0^t e^{-\l s}h(s)^2ds = {\l\over2} + \int_0^t e^{-\l s}h(s)^2ds. $$
Set for $t\in[0,k]$,
$$ H(t)\deq \int_0^t e^{-\l s}h(s)^2ds. $$
Then $H(0)=0$ and
$$ e^{\l t} H'(t) = h(t)^2 \les [{\l/2} + H(t)]^2, \q\forall t\in[0,k]. $$
or equivalently,
$$d\lt[{-1\over \l/2+H(t)}\rt] \les e^{-\l t}, \q\forall t\in[0,k]. $$
Integration gives
$${2\over\l} - {1\over \l/2+H(t)} \les \int_0^t e^{-\l s}ds \les {1\over \l}, \q\forall t\in[0,k], $$
from which we have
$$ h(t) \les \l,  \q\forall t\in[0,k]. $$
Thus (noting that $T-k<N+1$),
\begin{align}\label{Si:bound-1}
|\Si_{\scT}(t)|
&= {1\over K_3}e^{-\l(k-t)} h(k-t) \les {\l\over K_3}e^{-\l(k-t)} = {\l\over K_3}e^{\l(T-k)}e^{-\l(T-t)} \nn\\
&\les {\l\over K_3}e^{\l(N+1)}e^{-\l(T-t)}, \q\forall t\in[0,k].
\end{align}
For $t\in[k,T]$, we have $0\les T-t\les T-k \les N+1$. Thus,
\begin{align}\label{Si:bound-2}
|\Si_{\scT}(t)| &= |\Si_{\scT-t}(0)| \les K_4\deq \max_{s\in[0,N+1]}|\Si_{s}(0)| \nn\\
&\les K_4e^{\l(N+1)} e^{-\l(T-t)},  \q\forall t\in[k,T].
\end{align}
Combining \rf{Si:bound-1} and \rf{Si:bound-2} yields the desired \rf{Si:bound}.
It remains to show that \rf{Si:bound} also holds for $T<2N$ with a possibly
different constant $K>0$ that is independent of $T$. This can be proved in
a similar way as \rf{Si:bound-2}:
\begin{align*}
|\Si_{\scT}(t)| &= |\Si_{\scT-t}(0)| \les  K_5 \deq \max_{s\in[0,2N]}|\Si_{s}(0)| \nn\\
&\les \big(K_5e^{2N\l}\big) e^{-\l(T-t)},  \q\forall 0\les t\les T\les 2N.
\end{align*}
The proof is complete.
\end{proof}

\section{The Turnpike Property}\label{Sec:TP}

With the above preparations, we are ready to establish the turnpike property
for Problem (DG)$_{\scT}$. This will be done in the current section.
Denote by $\cP(\dbR^n)$ the set of probability distributions on $(\dbR^n,\sB(\dbR^n))$
having finite second moment, i.e., if $\mu\in\cP(\dbR^n)$, then $\int_{\dbR^n}|x|^2\mu(dx)<\i$.
We equip $\cP(\dbR^n)$ with the $L^2$-Wasserstein distance:
$$ d(\mu_1,\mu_2) \deq \inf\lt\{\sqrt{\dbE|\xi_1-\xi_2|^2}\Bigm|
\xi_i\text{ is an $\dbR^n$-valued random variable with }\mu_{\scp\xi_i}=\mu_i; ~i=1,2 \rt\},$$
where $\mu_{\scp\xi_i}$ denotes the probability distribution of $\xi_i$.
Consider the following SDE over $[0,\i)$:
\begin{equation}\label{SDE:TP-limit}
d\cX(t) = [\cA\cX(t)+\a]dt + [\cC\cX(t)+\b]dW(t),
\end{equation}
where $\cA,\cC\in\dbR^{n\times n}$, and $\a,\b\in\dbR^n$ are constant valued.

\ms

The following result, found in \cite[Proposition 4.2]{Sun-Yong2024JDE},
establishes the uniqueness and existence of a stationary solution to \rf{SDE:TP-limit}.

\begin{lemma}\label{lmm:SDE-stationary}
Suppose that the system $[\cA,\cC]$ is $L^2$-stable. Then there exists a unique
distribution $\mu$ with finite second moment such that the solution $\cX(\cd)$
of \rf{SDE:TP-limit} with initial distribution $\mu$ has a stationary distribution.
Namely, for any Borel set $\G$ in $\dbR^n$,
$$
\dbP[\cX(t)\in\G]=\mu(\G), \q\forall t\ges0.
$$
\end{lemma}

Next, consider, for each $T>0$, the following SDE over $[0,T]$:
\begin{equation}\label{SDE:TP}\left\{\begin{aligned}
d\cX(t) &= [\cA_{\scT}(t)\cX(t)+\a_{\scT}(t)]dt + [\cC_{\scT}(t)\cX(t)+\b_{\scT}(t)]dW(t), \\
 \cX(0) &= x,
\end{aligned}\right.\end{equation}
where $\cA_{\scT},\cC_{\scT}:[0,T]\to\dbR^{n\times n}$, and $\a_{\scT},\b_{\scT}:[0,T]\to\dbR^n$
are deterministic Lebesgue measurable, bounded functions.
We denote by $\cX_{\scT}(\cd\,;x)$ the solution of \rf{SDE:TP} over $[0,T]$ and by
$\mu_{\scT}(t;x)$ the probability distribution of $\cX_{\scT}(t;x)$.

\begin{proposition}\label{prop:SDE-TP}
Suppose that the system $[\cA,\cC]$ is $L^2$-stable, and let $\mu$ be the unique distribution
in \autoref{lmm:SDE-stationary}. If there exist constants $K,\l>0$, independent of $T$, such that
\begin{equation}\label{cdtn:SDE-TP}
|\cA_{\scT}(t)-\cA|+|\cC_{\scT}(t)-\cC|+|\a_{\scT}(t)-\a|+|\b_{\scT}(t)-\b|
\les K\[e^{-\l t}+ e^{-\l(T-t)}\], \q\forall t\in[0,T],
\end{equation}
then for any initial state $x\in\dbR^n$,
$$
d(\mu_{\scT}(t;x),\mu) \les K\big(|x|^2+1\big)\[e^{-\l t}+e^{-\l(T-t)}\], \q\forall t\in[0,T],
$$
with possibly different constants $K,\l>0$ that are independent of $x$ and $T$.
\end{proposition}

\begin{proof}
Let $\xi$ be an $\sF_0$-measurable, $\dbR^n$-valued random variable with distribution $\mu$,
and let $\cX(\cd)$ be the solution of \rf{SDE:TP-limit} with initial condition $\cX(0)=\xi$.
Then for any $t\ges0$, the distribution of $\cX(t)$ is identically equal to $\mu$.
Define, for $t\in[0,T]$,
\begin{align*}
& \cY_{\scT}(t;x)\deq\cX_{\scT}(t;x)-\cX(t),\\
& \h\a_{\scT}(t)\deq[\cA_{\scT}(t)-\cA]\cX(t)+\a_{\scT}(t)-\a,\\
& \h\b_{\scT}(t)\deq[\cC_{\scT}(t)-\cC]\cX(t)+\b_{\scT}(t)-\b.
\end{align*}
Then $\cY_{\scT}(0;x)=x-\xi$ and
$$ d\cY_{\scT}(t;x) = [\cA_{\scT}(t)\cY_{\scT}(t;x) + \h\a_{\scT}(t)]dt
+ [\cC_{\scT}(t)\cY_{\scT}(t;x)+\h\b_{\scT}(t)]dW(t). $$
Since $[\cA,\cC]$ is $L^2$-stable, by \cite[Theorem 3.2.3]{Sun-Yong2020book-a},
there exists a matrix $\cP\in\dbS^n_+$ such that
$$ \cP\cA+\cA^\top\cP+\cC^\top\cP\cC=-I_n. $$
By applying It\^{o}'s rule to $t\mapsto\lan\cP\cY_{\scT}(t;x),\cY_{\scT}(t;x)\ran$, we have
\begin{equation}\label{EVT=}\begin{aligned}
& {d\over dt}\dbE\lan\cP\cY_{\scT}(t;x),\cY_{\scT}(t;x)\ran
  =\dbE\Big\{2\lan\cP\cY_{\scT}(t;x),\cA_{\scT}(t)\cY_{\scT}(t;x)+\h\a_{\scT}(t)\ran \\
&\qq+\lan\cP[\cC_{\scT}(t)\cY_{\scT}(t;x)+\h\b_{\scT}(t)],\cC_{\scT}(t)\cY_{\scT}(t;x)
    +\h\b_{\scT}(t)\ran\Big\} \\
&\q=\dbE\Big\{\lan[\cP\cA_{\scT}(t)+\cA_{\scT}(t)^\top\cP+\cC_{\scT}(t)^\top\cP\cC_{\scT}(t)]
    \cY_{\scT}(t;x),\cY_{\scT}(t;x)\ran \\
&\qq+2\lan\cY_{\scT}(t;x),\cP\h\a_{\scT}(t)+\cC_{\scT}(t)^\top\cP\h\b_{\scT}(t)\ran
    +\lan\cP\h\b_{\scT}(t),\h\b_{\scT}(t)\ran \Big\}.
\end{aligned}\end{equation}
To estimate the right-hand side of \rf{EVT=}, we first observe that
\begin{align*}
& \cP\cA_{\scT}(t)+\cA_{\scT}(t)^\top\cP+\cC_{\scT}(t)^\top\cP\cC_{\scT}(t) \\
&\q= \cP\cA+\cA^\top\cP+\cC^\top\cP\cC + \cP[\cA_{\scT}(t)-\cA]+[\cA_{\scT}(t)-\cA]^\top\cP \\
&\q\hp{=\ } +[\cC_{\scT}(t)-\cC]^\top\cP\cC_{\scT}(t)+\cC^\top\cP[\cC_{\scT}(t)-\cC] \\
&\q= -I_n + \cP[\cA_{\scT}(t)-\cA]+[\cA_{\scT}(t)-\cA]^\top\cP+[\cC_{\scT}(t)-\cC]^\top\cP\cC_{\scT}(t)
     +\cC^\top\cP[\cC_{\scT}(t)-\cC],
\end{align*}
which, together with \rf{cdtn:SDE-TP}, implies that
\begin{equation}\label{sub:1}\begin{aligned}
& \dbE\lan[\cP\cA_{\scT}(t)+\cA_{\scT}(t)^\top\cP+\cC_{\scT}(t)^\top\cP
  \cC_{\scT}(t)]\cY_{\scT}(t;x),\cY_{\scT}(t;x)\ran\\
&\q\les [-1+K\psi(t,T)]\,\dbE|\cY_{\scT}(t;x)|^2, \q\forall t\in[0,T]
\end{aligned}\end{equation}
for some constant $K>0$ independent of $T$ and $x$, where
$$ \psi(t,T)\deq e^{-\l t}+e^{-\l(T-t)}, \q t\in[0,T]. $$
Further, since $\dbE|\cX(t)|^2$ is finite and time-invariant by \autoref{lmm:SDE-stationary},
we have (for possibly a different $K>0$ that is independent of $T$ and $x$)
\begin{align*}
\dbE|\h\a_{\scT}(t)|^2
&=\dbE\big|[\cA_{\scT}(t)-\cA]\cX(t)+\a_{\scT}(t)-\a\big|^2 \les K\psi(t,T), \q\forall t\in[0,T], \\
\dbE|\h\b_{\scT}(t)|^2
&=\dbE\big|[\cC_{\scT}(t)-\cC]\cX(t)+\b_{\scT}(t)-\b\big|^2 \les K\psi(t,T), \q\forall t\in[0,T].
\end{align*}
Consequently,
\begin{equation}\label{sub:2}\begin{aligned}
& \dbE\[2\lan\cY_{\scT}(t;x),\cP\h\a_{\scT}(t)+\cC_{\scT}(t)^\top\cP\h\b_{\scT}(t)\ran
  +\lan\cP\h\b_{\scT}(t),\h\b_{\scT}(t)\ran\]\\
&\q\les \dbE\lt[{1\over2}|\cY_{\scT}(t;x)|^2+2|\cP\h\a_{\scT}(t)+\cC_{\scT}(t)^\top\cP\h\b_{\scT}(t)|^2
        +\lan\cP\h\b_{\scT}(t),\h\b_{\scT}(t)\ran\rt]\\
&\q\les {1\over2}\dbE|\cY_{\scT}(t;x)|^2+K\psi(t,T), \q\forall t\in[0,T]
\end{aligned}\end{equation}
for some constant $K>0$ independent of $T$ and $x$.
Let $\g_1>0$ and $\g_2>0$ be the smallest and largest eigenvalues of $\cP$, respectively.
Substituting \rf{sub:1}--\rf{sub:2} into \rf{EVT=} yields
\begin{equation}\label{EVT<=}\begin{aligned}
{d\over dt}\dbE\lan\cP\cY_{\scT}(t;x),\cY_{\scT}(t;x)\ran
&\les \[-{1\over2}+K\psi(t,T)\]\dbE|\cY_{\scT}(t;x)|^2+K\psi(t,T) \\
&\les \[-{1\over2\g_1}+{K\over\g_1}\psi(t,T)\]\dbE\lan\cP\cY_{\scT}(t;x),\cY_{\scT}(t;x)\ran + K\psi(t,T) \\
&\les \big[-\rho+K\psi(t,T)\big]\dbE\lan\cP\cY_{\scT}(t;x),\cY_{\scT}(t;x)\ran + K\psi(t,T),
\end{aligned}\end{equation}
where the constant $K>0$ at the end might be different, and $\rho<\min\{{1\over2\g_2},\l\}$.
By the differential form of the Gronwall inequality,
\begin{align*}
\dbE\lan\cP\cY_{\scT}(t;x),\cY_{\scT}(t;x)\ran
&\les \dbE\lan\cP\cY_{\scT}(0;x),\cY_{\scT}(0;x)\ran\exp\lt(\int_0^t\big[-\rho+K\psi(s,T)\big]ds\rt) \\
&\hp{=\ } +\int_0^tK\psi(s,T)\exp\lt(\int_s^t\big[-\rho+K\psi(r,T)\big]dr\rt)ds.
\end{align*}
Note that
$$\int_s^t\big[-\rho+K\psi(r,T)\big]dr \les -\rho(t-s)+{2K\over\l} $$
leading to
$$\exp\lt(\int_s^t\big[-\rho+K\psi(r,T)\big]dr\rt) \les e^{2K\over\l}e^{-\rho(t-s)},
\q\forall 0\les s\les t\les T.$$
Also, we have (recalling $\rho<\l$)
\begin{align*}
& \int_0^tK\psi(s,T)\exp\lt(\int_s^t\big[-\rho+K\psi(r,T)\big]dr\rt)ds \\
&\q\les Ke^{2K\over\l}\int_0^t\(e^{-\l s}+e^{-\l(T-s)}\)e^{-\rho(t-s)}ds \\
&\q= Ke^{2K\over\l}\lt[{1\over\l-\rho}\big(e^{-\rho t}-e^{-\l t}\big)
     + {e^{-\l T}\over\l+\rho}\big(e^{\l t}-e^{-\rho t}\big)\rt]  \\
&\q\les Ke^{2K\over\l}\lt[{1\over\l-\rho}e^{-\rho t} + {1\over \l+\rho}e^{-\l(T-t)}\rt] \\
&\q\les K\[e^{-\rho t}+e^{-\rho(T-t)}\],
\end{align*}
where in the last inequality, $K>0$ is a possible different constant,
independent of $T>0$ and $x\in\dbR^n$. Consequently,
\begin{align*}
& \dbE|\cX_{\scT}(t;x)-\cX(t)|^2 = \dbE|\cY_{\scT}(t;x)|^2
  \les {1\over\g_1}\dbE\lan\cP\cY_{\scT}(t;x),\cY_{\scT}(t;x)\ran \\
&\q\les K\[\dbE\lan\cP(x-\xi),x-\xi\ran e^{-\rho t}+\big(e^{-\rho t}+e^{-\rho(T-t)}\big)\] \\
&\q\les K\[\big(|x|^2+\dbE|\xi|^2\big)e^{-\rho t}+\big(e^{-\rho t}+e^{-\rho(T-t)}\big)\] \\
&\q\les K\big(|x|^2+1\big)\[e^{-\rho t} + e^{-\rho(T-t)}\], \q\forall t\in[0,T]
\end{align*}
for some constant $K>0$ that is independent of $T$ and $x$. It follows that
$$ d(\mu_{\scT}(t;x),\mu)\les\sqrt{\dbE|\cX_{\scT}(t;x)-\cX(t)|^2 }
\les K(|x|+1)\[e^{-\rho t/2}+e^{-\rho(T-t)/2}\],\q\forall t\in[0,T]. $$
This completes the proof.
\end{proof}

Let (A1)--(A2) hold. Let $P\in\dbS^n$ be the unique solution to the ARE \rf{ARE:game}
such that the conditions (i)--(ii) in \autoref{thm:ARE-uniqueness} hold.
Let $\Th$ be defined as in \rf{def:Th}. Let $\f$ be the solution to the algebra equation
\begin{equation}\label{ME:phi}
(A+B\Th)^\top\f + (C+D\Th)^\top P\si + \Th^\top r + Pb + q = 0,
\end{equation}
that is,
$$\f = -[(A+B\Th)^\top]^{-1}[(C+D\Th)^\top P\si + \Th^\top r + Pb + q], $$
and define
\begin{equation}\label{def:v}
v\deq-(R+D^\top PD)^{-1}(B^\top\f+D^\top P\si+r).
\end{equation}
Consider the following SDE over $[0,\i)$:
\begin{equation}\label{SDE:TP-LIM}
dX(t) = [(A+B\Th)X(t)+Bv+b]dt+[(C+D\Th)X(t)+Dv+\si]dW(t).
\end{equation}
Since $[A+B\Th,C+D\Th]$ is $L^2$-stable (see \autoref{thm:PTgoP}), by \autoref{lmm:SDE-stationary},
there exists a unique initial distribution $\mu^*$ with finite second moment such that
the corresponding solution $X^*(\cd)$ of \rf{SDE:TP-LIM} has a stationary distribution.
This means
\begin{equation}\label{m*}
\dbP(X(t)\in\G)=\mu^*(\G), \q\forall\G\in\sB(\dbR^n).
\end{equation}
Let
\begin{equation}\label{def:TP-LIM}
u^*(t) \deq \Th X^*(t)+v,\q t\ges0,
\end{equation}
and denote the (stationary) distribution of $u^*(t)$ by $\nu^*$.
Also, denote by $\mu_{\scT}(t,x)$ and $\nu_{\scT}(t,x)$ the distributions
of the optimal state process $\bX(t)$ and the open-loop saddle strategy
$\bar u_{\scT}(t)\deq\begin{pmatrix}\bar u_{1,\scT}(t) \\ \bar u_{2,\scT}(t)\end{pmatrix}$
of Problem (DG)$_{\scT}$ for the initial state $x$, respectively.
We now present the following result, which establishes the exponential turnpike
property for Problem (DG)$_{\scT}$.

\begin{theorem}\label{thm:TP-main}
Let {\rm(A1)--(A2)} hold. Then there exist constants $K,\l>0$,
independent of $x$ and $T$, such that
$$ d(\mu_{\scT}(t;x),\mu^*) + d(\nu_{\scT}(t;x),\nu^*)
\les K\big(|x|^2+1\big)\[e^{-\l t} + e^{-\l(T-t)}\], \q\forall t\in[0,T]. $$
\end{theorem}

\begin{proof}
By \autoref{thm:uT-rep},
\begin{equation}\label{rep:uT}
\bar u_{\scT}(t)=\bar\Th_{\scT}(t)\bar X_{\scT}(t)+\bar v_{\scT}(t),\q t\in[0,T],
\end{equation}
where $\bar\Th_{\scT}(\cd)$ and $\bar v_{\scT}(\cd)$ are defined
by \rf{def:barTh} and \rf{def:vT}, respectively.
Upon substitution of \rf{rep:uT} into the state equation \rf{state},
we see that $\bX(\cd)$ satisfies the following closed-loop system:
\begin{equation}\label{cloop*}\left\{\begin{aligned}
d\bX(t) &= \big\{[A+B\bar\Th_{\scT}(t)]\bX(t)+B\bar v_{\scT}(t)+b\big\}dt \\
        &\hp{=\ } + \big\{[C+D\bar\Th_{\scT}(t)]\bX(t)+D\bar v_{\scT}(t)+\si\big\}dW(t), \q t\in[0,T], \\
 \bX(0) &= x.
\end{aligned}\right.\end{equation}
According to \autoref{crllry:unibound:R+DPD} and \autoref{thm:PT-P},
there exist constants $K,\l>0$, independent of $T$ such that
\begin{align*}
|\Th-\bar\Th_{\scT}(t)|
&= \big|[R+D^\top P_{\scT}(t)D]^{-1}\big\{B^\top[P-P_{\scT}(t)]+D^\top[P-P_{\scT}(t)](C+D\Th)\big\}\big| \\
&\les \big|[R+D^\top P_{\scT}(t)D]^{-1}\big|\big(|B|+|D||C+D\Th|\big)|P-P_{\scT}(t)| \\
&\les Ke^{-\l(T-t)}, \q\forall t\in[0,T].
\end{align*}
To estimate $|\bar v_{\scT}(t)-v|$, let
$$ \phi_{\scT}(t) \deq \f_{\scT}(T-t), \q t\in[0,T]. $$
where $\f_{\scT}(\cd)$ is the solution to the ODE \rf{ODE:phi}. Then
\begin{equation}\label{ODE:phi*}\left\{\begin{aligned}
\dot\phi_{\scT}(t)
&= [A+B\bar\Th_{\scT}(T-t)]^\top\phi_{\scT}(t) + [C+D\bar\Th_{\scT}(T-t)]^\top P_{\scT}(T-t)\si \\
&\hp{=\ } +\bar\Th_{\scT}(T-t)^\top r + P_{\scT}(T-t)b + q, \q t\in[0,T], \\
\phi_{\scT}(0) &= 0.
\end{aligned}\right.\end{equation}
Comparing \rf{ODE:phi*} and \rf{ME:phi}, and then applying \autoref{prop:SDE-TP} to
the case of ordinary differential equations yields
$$|\f-\f_{\scT}(t)|=|\f-\phi_{\scT}(T-t)|\les K\[e^{-\l t}+e^{-\l(T-t)}\], \q\forall t\in[0,T] $$
for some constants $K,\l>0$ independent of $T$. As a consequence,
$$|v-\bar v_{\scT}(t)|\les K\[e^{-\l t}+e^{-\l(T-t)}\],\q\forall t\in[0,T],$$
with possibly different constants $K,\l>0$ that are independent of $T$.
Comparing \rf{cloop*} and \rf{SDE:TP-LIM}, and the applying \autoref{prop:SDE-TP}
again yields the desired result.
\end{proof}

\end{document}